\documentclass[12pt,a4paper]{article}
\usepackage{mathbbol}
\usepackage{amsfonts}
\usepackage{microtype}
\usepackage{amsfonts,amssymb,amsfonts,amsthm,amsmath}
\usepackage{cleveref}
\usepackage{mathrsfs}
\usepackage{enumerate}
\usepackage{graphicx,,epsfig,a4,url}
\usepackage[all]{xy}
\usepackage{color,xcolor}

\newtheorem{thm}{Theorem}[section]
\newtheorem{cor}[thm]{Corollary}
\newtheorem{lem}[thm]{Lemma}

\newtheorem{prop}[thm]{Proposition}
\theoremstyle{definition}
\newtheorem{defn}[thm]{Definition}

\newtheorem{conj}[thm]{Conjecture}
\theoremstyle{remark}
\newtheorem{rem}[thm]{Remark}

\numberwithin{equation}{section}
\newcommand{\supp}{\textnormal{Supp}}

\title {The equivariant coarse Novikov conjecture and coarse embedding\footnote{The first author is supported by NSFC (No. 11871342, 11771143). The second author is supported by NSFC(No. 11771061). The third author is supported by NSF(No. 1564398, 1700021).}}
\author{Benyin Fu \and Xianjin Wang \and Guoliang Yu}
\date{}
\begin{document}
\baselineskip=16pt
\maketitle

\begin{abstract}
The equivariant coarse Novikov conjecture provides an algorithm for determining nonvanishing of equivariant higher index of elliptic differential operators on noncompact manifolds. In this article, we prove the equivariant coarse Novikov conjecture under certain coarse embeddability conditions. More precisely, if a discrete group $\Gamma$ acts on a bounded geometric space $X$ properly, isometrically, and with bounded distortion, $X/\Gamma$ and $\Gamma$ admit coarse embeddings into Hilbert space, then the $\Gamma$-equivariant coarse Novikov conjecture holds for $X$. Here bounded distortion means that for any $\gamma\in\Gamma$, $\sup_{x\in Y} d(\gamma x,x)<\infty$, where $Y$ is a fundamental domain of the $\Gamma$-action on $X$.
\end{abstract}

\section{Introduction}

Let $X$ be a proper metric space, let $\Gamma$ be a countable discrete group. Assume that $\Gamma$ acts on $X$ properly and isometrically. One can define an equivariant higher index map (see~\cite{Connes,RYBook,Shan})
$$Ind_{\Gamma}:\lim\limits_{d\rightarrow \infty}K_*^\Gamma(P_d(X))\rightarrow K_*(C^*(X)^\Gamma),$$
where $K_*^{\Gamma}(P_d(X))$ is the $\Gamma$-equivariant $K$-homology group of Rips complex $P_d(X)$ and $C^*(X)^{\Gamma}$ is the equivariant Roe algebra, $K_*(C^*(X)^{\Gamma})$ is the receptacle of higher index for elliptic differential operators.
The equivariant coarse Novikov conjecture states that the equivariant higher index map is injective. This conjecture provides an algorithm for determining the nonvanishing of equivariant higher index of elliptic differential operators. Nonvanishing of equivariant higher index has important applications to geometry and topology such as the positive scalar curvature problem (see~\cite{GWY, GongYu, Ros, sherry, HR, Tang}). When the $\Gamma$-action is cocompact, i.e., $X/\Gamma$ is compact,
$C^*(X)^\Gamma$ is Morita equivalent to $C_r^*(\Gamma)$, the reduced $C^*$-algebra of $\Gamma$, then the equivariant
higher index map is the Baum-Connes map introduced by Baum, Connes and Higson (see~\cite{BC,BC94}).  When $\Gamma$ is
trivial, the equivariant higher index map is the coarse Baum-Connes map introduced by Roe, Higson and Yu (see~\cite{HR, Ro93, Ro, Yu95-2,DA,FO1,FO2,FO3}).

Gromov~\cite{Gro} introduced the following definition of coarse embedding.

\begin{defn}\label{sec:1:def:1}
Let $X$ and $Y$ be two metric spaces. A map $f:X\rightarrow Y$ is said to be a {\it coarse embedding} if there exist non-decreasing functions $\rho_+, \rho_-$ from $\mathbb{R}_+:=[0,+\infty)$ to $\mathbb{R}_+$ such that
\begin{enumerate}
\item[(1)] $\rho_-(d(x,y))\leq \|f(x)-f(y)\|\leq \rho_+(d(x,y))$ for all $x, y\in X$;
\item[(2)] $\lim\limits_{r\rightarrow +\infty}\rho_{\pm}(r)\rightarrow +\infty$.
\end{enumerate}
\end{defn}
Coarse embeddability into Hilbert space implies the coarse Novikov conjecture~\cite{Yu}. This result can be extended to spaces coarsely embeddable into other spaces with certain uniform geometric properties~\cite{CWY, KY, Shan, ShanWang}. We also remark that the coarse Novikov conjecture still holds for certain expanders which are not coarsely embeddable into Hilbert space~\cite{GWY2,GongWangYu, OyonoYu, Ruffus1, Ruffus2}.

Recall that a discrete metric space $X$ is said to have {\it bounded geometry} if for any $r>0$,
there exists $N>0$, such that the ball in $X$ with radius $r$ contains at most $N$ elements. A subset $Y$ of $X$ is called a {\it fundamental domain} for the $\Gamma$-action on $X$ if $X$ is the disjoint union $X=\sqcup_{\gamma\in\Gamma}\gamma Y$.
We say that $X$ has {\it bounded distortion} if  for any $\gamma\in\Gamma$, 
$$\sup\limits_{x\in Y}d(\gamma x,x)<+\infty,$$  where $Y$ is a fundamental domain for $X$. 
In this paper, we want to prove the following result:

\begin{thm}\label{sec:1:thm:1.2}
Let $\Gamma$ be a countable discrete group, let $X$ be a discrete metric space with bounded geometry. Assume that $\Gamma$ acts on $X$ properly, isometrically, and with bounded distortion.
If both $X/\Gamma$ and $\Gamma$ admit coarse embeddings into a Hilbert space $H$, then the equivariant higher index map $$Ind_{\Gamma}:\lim\limits_{d\rightarrow \infty}K_*^\Gamma(P_d(X))\rightarrow K_*(C^*(X)^\Gamma)$$ is injective.
\end{thm}

When the group $\Gamma$ is trivial, this theorem has been proved by Yu in~\cite{Yu}. In fact, he proved that this index map is an isomorphism. When $X=\Gamma$, this conclusion has been proved by Skandalis, Tu and Yu in~\cite{Skan} by using the Higson's descent technique in~\cite{Hi}.

Theorem \ref{sec:1:thm:1.2} has important application to the existence of $\Gamma$-invariant Riemannian metric with positive scalar curvature for a Riemannian manifold. Let $M$ be
a spin Riemannian manifold with a proper and isometrical action of a discrete group $\Gamma$. 
Recall that a discrete subset $X\subset M$ is called a $\varepsilon$-{\it net} for $M$ if $d(x,x')\geq \varepsilon$ for any $x,x'\in X$ and for all $y\in M$ there is $x\in X$ with $d(y,x)<\varepsilon$.
Assume that $M$ has a $\Gamma$-invariant $\varepsilon$-net $X$ with bounded geometry.
Let $\{\varphi_x\}_{x\in X}$ be a partition of unity subordinate to a $\Gamma$-equivariant open cover $\{U_x\}_{x\in X}$ of $M$ with $x\in U_x$ and $\sup(\text{diam}(U_x))<\infty$, where the ``$\Gamma$-equivariance" of $\{U_x\}_{x\in X}$ means that if $U_{x_0}\in\{U_x\}_{x\in X}$ and $\gamma\in\Gamma$, then $\gamma U_{x_0}\in\{U_x\}_{x\in X}$.
In this case, we can define a $\Gamma$-invariant
map $$\phi:M\rightarrow P_d(X),\quad z\mapsto \sum\limits_{x\in X}\varphi_x(z)x$$ when $d$ is large enough. The map $\phi$ induces a homomorphism
$\phi_*:K_*^\Gamma(M)\rightarrow K_*^\Gamma(P_d(X))$, where $d$ is large enough. If $D$ is a $\Gamma$-invariant Dirac operator on $M$, then $Ind_\Gamma(\phi_*([D]))$ is called the {\it equivariant higher index of $D$}. If this equivariant higher index is nonzero, then $M$ has no $\Gamma$-invariant Riemannian metric with uniformly positive scalar curvature by the Lichnerowicz formula. So we have the following corollary.

\begin{cor}
\label{c:positive}
Let $M$ be a complete spin manifold with bounded geometry and let $\Gamma$ act on $M$ properly, isometrically and with bounded distortion. Assume that $X$ is a $\Gamma$-invariant $\varepsilon$-net of $M$. If $M/\Gamma$ and $\Gamma$ admit coarse embeddings into a Hilbert space, $\phi_*([D])\neq 0$ in $\lim\limits_{d\rightarrow \infty}K_*^\Gamma(P_d(X))$, then $M$ can not have  any $\Gamma$-invariant Riemannian metric with uniformly
positive scalar curvature.
\end{cor}

We remark that $\lim\limits_{d\rightarrow \infty}K^\Gamma_*(P_d(X))$ and $\phi_*([D])$ in the above Corollary~\ref{c:positive} are computable. Hence the injectivity of $Ind_{\Gamma}$ provides an algorithm of determining nonvanishing of the equivariant higher index of $D$.

The paper is organized as follows. In section 2, we recall the equivariant higher index map for a metric space with proper and isometrical group action. In section 3, we define an equivariant $\pi$-localization algebra. There is an evaluation map from this equivariant $\pi$-localization algebra to the equivariant Roe algebra. This evaluation map induces a homomorphism at their $K$-theory level. We prove that this homomorphism is injective. The proof for this uses the twisted equivariant Roe algebras and twisted equivariant $\pi$-localization algebras, which are constructed in a way similar to those in~\cite{Yu} by using the coarse embedding of the quotient space $X/\Gamma$. In section 4, we introduce two variations of the twisted equivariant localization algebras and the twisted equivariant $\pi$-localization algebras by using the coarse embeddings of $\Gamma$ and $X/\Gamma$. We discuss the relations between their K-theory level. In section 5, we introduce two Bott maps by using the coarse embeddings of $\Gamma$ and $X/\Gamma$ and give the proof of the main theorem.

\section{The equivariant higher index map}
In this section, we shall first recall the concepts of equivariant $K$-homology and equivariant Roe algebras and then construct the equivariant higher index map.

Let $X$ be a proper metric space (a metric space is called
$proper$ if every closed ball is compact) and let $\Gamma$ be a countable discrete group acting on $X$ properly and isometrically.
The action of $\Gamma$ on $X$ is $proper$ in the sense that the map
$$X\times \Gamma\rightarrow X\times X, \qquad (x,\gamma)\mapsto (\gamma x,x)$$
is proper, i.e., the preimage of a compact set is compact. $\Gamma$ acts on $X$ {\it
isometrically} if $d(x,y)=d(\gamma x,\gamma y)$ for all $\gamma\in\Gamma$ and $x,y\in X$. In this case, we call $X$ a {\it $\Gamma$-space}. An
$X$-$module$ is a separable Hilbert space $H$ equipped with a faithful and non-degenerate $*$-
representation $\phi:C_0(X)\rightarrow B(H)$ whose range contains no nonzero compact operators, where $C_0(X)$ is
the algebra of all complex valued continuous functions on $X$ which vanish at infinity.

\begin{defn}
Let $X$ be a $\Gamma$-space. For $\gamma \in \Gamma$ and $f\in C_0(X)$, define $\gamma(f)\in C_0(X)$ by $$\gamma(f)(x)=f(\gamma^{-1}x).$$
\end{defn}

\begin{defn}
Let $H$ be an $X$-module. We say that $H$ is a {\it covariant $X$-module} if it is equipped with a
unitary action $\rho$ of $\Gamma$, i.e., $\rho:\Gamma\rightarrow \mathcal {U}(H)$ is a group
homomorphism from $\Gamma$ to the set of unitary elements in $B(H)$, compatible with the action
of $\Gamma$ on $X$ in the sense that for $v\in H$, $T\in B(H)$ and $\gamma\in \Gamma$,
$$\gamma(T)(v)=\rho(\gamma)T\rho(\gamma)^*(v),$$
we call such a triple $(C_0(X), \Gamma, \phi)$ a $covariant\ system$.
\end{defn}

\begin{defn}[\cite{KY}]
\label{d:ad_rep}
A covariant system $(C_0(X), \Gamma, \phi)$ is $admissible$ if there exists a $\Gamma$-Hilbert space $H_X$ and a separable and infinite dimensional $\Gamma$-Hilbert space $H'$ such that
\begin{enumerate}
\item[(1)] $H$ is isomorphic to $H_X\otimes H'$;
\item[(2)] $\varphi=\varphi_0\otimes I$ for some $\Gamma$-equivariant $*$-homomorphism $\varphi_0$ from $C_0(X)$ to $B(H_X)$ such that $\varphi_0(f)$ is not in $K(H_X)$ for any nonzero function $f\in C_0(X)$ and $\varphi_0$ is nondegenerate in the sense that $\{\varphi_0(f)H_X:f\in C_0(X)\}$ is dense in $H_X$;
\item[(3)] for each finite subgroup $F$ of $\Gamma$ and every $F$-invariant Borel subset $E$ of $X$, there exists a trivial $F$-representation $H_E$ such that $$\chi_E H'\cong l^2(F)\otimes H_E.$$
\end{enumerate}

\end{defn}
In the above definition, the $F$-action on $l^2(F)$ is regular, i.e., $(\gamma \xi)(z)=\xi(\gamma^{-1}z)$ for every $\gamma\in F, \xi\in l^2(F)$ and $z\in F$.

\begin{defn}
Let $(C_0(X),\Gamma, \phi)$ be an admissible covariant system.
\begin{enumerate}
\item[(1)] The {\it support} $\supp(T)$ of a bounded linear operator $T:H\rightarrow H$ is the complement of the set of points $(x,y)\in X\times X$ for which there exist $f$ and $f'$ in $C_0(X)$ such that $$\phi(f')T\phi(f)=0,\quad f(x)\neq 0,\quad f'(y)\neq 0.$$
\item[(2)] A bounded operator $T:H\rightarrow H$ has finite propagation if $$\text{sup}\{d(x,y):(x,y)\in \supp(T)\}<\infty.$$
This number is called the {\it propagation of $T$}.
\item[(3)] A bounded operator $T:H\rightarrow H$ is {\it locally compact} if the operators $\phi(f)T$ and $T\phi(f)$ are compact for all $f\in C_0(X)$.
\item[(4)] A bounded operator $T:H\rightarrow H$ is {\it $\Gamma$-invariant} if $\gamma(T)=T$ for all $\gamma\in \Gamma$.
\end{enumerate}
\end{defn}
\begin{defn}
The {\it equivariant Roe algebra} $C^*(X)^{\Gamma}$ is the operator norm closure of the $*$-algebra of all locally compact and $\Gamma$-invariant operators with finite propagation.
\end{defn}

%\begin{prop}[\cite{Shan,Yu95-2}]
%\label{sec:2:prop:eqiv-dep}
%The equivariant Roe algebra $C^*(X, H_X)^{\Gamma}$ does not depend on the choice of $X$-module $H_X$.
%\end{prop}

%For this reason, $C^*(X, H_X)^\Gamma$ can be abbreviated as $C^*(X)^\Gamma$.

The equivariant Roe algebra $C^*(X)^{\Gamma}$ is Morita equivalent to the reduced group $C^*$-algebra $C_r^*(\Gamma)$ when $X/\Gamma$ is compact.
This result is essentially due to John Roe.

\begin{prop}[{\cite[Lemma~5.14]{Ro}}]
\label{morita-roe}
Let $(C_0(X),\Gamma, \phi)$ is an admissible covariant system, moreover if the $\Gamma$-action on $X$ is cocompact, then $C^*(X)^\Gamma$ is $*$-isomorphic to $C_r^*(\Gamma)\otimes \mathcal {K}$, where $C^*_r(\Gamma)$ is the reduced group $C^*$-algebra and $\mathcal {K}$ is the algebra of all compact operators on a separable and infinite dimensional Hilbert space.
\end{prop}

%The following definition is due to Kasparov (\cite{Ka,Yu95-2}).
Let us briefly recall the definition of equivariant $K$-homology introduced by Kasparov~\cite{Ka,Yu95-2} for a covariant system $(C_0(X),\Gamma, \phi)$.
\begin{defn}
$K_i^\Gamma(X)=KK_i^\Gamma(C_0(X),\mathbb{C})$ for $i=0,1$.
\end{defn}

To be more precise, $K_i^\Gamma(X)=KK_i^\Gamma(C_0(X),\mathbb{C})$ ($i=0,1$) are generated by certain cycles module a certain equivalent relation:

\begin{enumerate}
\item[(1)] a cycle for $K_0^\Gamma(X)$ is a triple $(H,\phi,F)$, where $\phi$ is the covariant representation from $C_0(X)$ to $B(H)$ and $F$ is a $\Gamma$-invariant bounded linear operator acting on $H$ such that $\phi(f) F-F\phi(f)$, $\phi(f)(FF^*-I)$ and $\phi(f)(F^*F-I)$ are in $K(H)$ for all $f\in C_{0}(X), \gamma\in \Gamma$;
\item [(2)] a cycle for $K_1^\Gamma(X)$ is a triple $(H,\phi,F)$, where $F$ is a $\Gamma$-invariant and self-adjoint operator acting on $H$ such that $\phi(f)(F^2-I)$, and $\phi(f) F-F\phi(f)$ are compact for all $f\in C_0(X), \gamma\in \Gamma$.
\end{enumerate}
In both cases, the equivalence relation on cycles is given by homotopy of operator $F$.

We remark that every class in $K^{\Gamma}_*(X)$ is equivalent to a cycle $(H,\phi,F)$ such that $(C_0(X),\Gamma,\phi)$ is an admissible covariant system~\cite{KY}. This can be seen as follows. We define a new $K$-homology group $\tilde{K}_*^{\Gamma}(X)$ by using cycles such that $(C_0(X),\Gamma,\phi)$ is an admissible covariant system. By the proof of~\cite[Proposition 5.5]{KY03} by Kasparov and Skandalis, we can show that there exists a $\Gamma$-Hilbert space $H_X$ satisfying the conditions of Definition~\ref{d:ad_rep} such that $H\oplus (H_X\hat{\otimes}l^2(\Gamma))$ is isomorphic to $H_X\hat{\otimes}l^2(\Gamma)$ as $\Gamma$-Hilbert space. Using this stabilization result, we can prove that the nature homomorphism from $\tilde{K}_*^{\Gamma}(X)$ to $K_*^{\Gamma}(X)$ is an isomorphism.

For any cycle $(H,\phi, F)$ in $K_0^\Gamma(X)$, let $\{U_i\}_{i\in I}$ be a locally finite, $\Gamma$-equivariant and uniformly bounded open cover of $X$. Let $\{\psi_i\}_{i\in I}$ be a $\Gamma$-equivariant partition of unity subordinate to $\{U_i\}_{i\in I}$. Define $$F'=\sum\limits_{i\in I}\phi(\sqrt{\psi_i})F\phi(\sqrt{\psi_i}),$$
where the infinite sum converges in strong topology. There is no difficulty to check that $(H,\phi, F')$ is equivalent to $(H, \phi, F)$ in $K^\Gamma_0(X)$. Note that $F'$ has finite propagation so that $F'$ is a multiplier of $C^*(X)^\Gamma$ and is invertible modulo $C^*(X)^\Gamma$. Hence $F'$ gives rise to an element in $K_0(C^*(X)^\Gamma)$, denoted by $\partial([F'])$. So we define the equivariant index map $Ind_\Gamma$ from $K^\Gamma_0(X)$ to $K_0(C^*(X)^\Gamma)$.  Similarly, we can define the equivariant index map from $K^\Gamma_1(X)$ to $K_1(C^*(X)^\Gamma)$.

\begin{defn}
Let $X$ be a discrete $\Gamma$-space with bounded geometry. For any $d>0$, the Rips complex $P_d(X)$ is the simplicial polyhedron whose set of vertices equals to $X$ and where a finite subset $Y\subseteq X$ spans a simplex if and only if the distance between any two points in $Y$ is less than or equal to $d$.
\end{defn}
For each $d>0$, the proper and isometrical action on $X$ induces a proper and isometrical action on $P_d(X)$ by $\gamma \cdot \sum\limits_{i=1}^k c_i x_i=\sum\limits_{i=1}^k c_i( \gamma\cdot x_i)$ for $x_i\in X, i=1,2,\cdots, k$ and $\sum\limits_{i=1}^k c_i=1$. Endow $P_d(X)$ with the spherical metric. Recall that on each path connected component of $P_d(X)$, the spherical metric is the maximal metric whose restriction to each simplex $\{\sum\limits_{i=0}^n t_i \gamma_i:t_i\geq 0, \sum\limits_{i=0}^n t_i =1\}$ is the metric obtained by identifying the simplex with
$$S_+^n:=\Big\{(s_0,s_1,\cdots, s_n)\in \mathbb{R}^{n+1}:s_i\geq 0, \sum\limits_{i=0}^n s_i^2 =1\Big\}$$
via the map
$$\sum\limits_{i=0}^n t_i \gamma_i\mapsto \Big(\frac{t_0}{\sqrt{\sum_{i=0}^n t_i^2}}, \frac{t_1}{\sqrt{\sum_{i=0}^n t_i^2}},\cdots,\frac{t_n}{\sqrt{\sum_{i=0}^n t_i^2}}\Big),$$
where $S_+^n$ is endowed with the standard Riemannian metric. The distance of a pair of points in different connected components of $P_d(X)$ is defined to be infinity.

The following conjecture is called the equivariant coarse Novikov conjecture:
\begin{conj}
If $X$ is a discrete metric space with bounded geometry and $\Gamma$ acts on $X$ properly and isometrically, then the equivariant index map $$Ind_\Gamma:\lim\limits_{d\rightarrow\infty}K_*^\Gamma(P_d(X))\rightarrow \lim\limits_{d\rightarrow\infty}K_*(C^*(P_d(X))^\Gamma)\cong K_*(C^*(X)^\Gamma)$$ is  injective.
\end{conj}

Now we shall recall the definition of the localization algebra $C_L^*(X)^\Gamma$ for a $\Gamma$-space $X$.
\begin{defn}
The {\it equivariant localization algebra} $C_L^*(X)^\Gamma$ is the norm-closure of the algebra of all bounded and uniformly norm-continuous functions $$f:[0,+\infty)\rightarrow C^*(X)^\Gamma$$ such that
$$\text{propagation}(f(t))\rightarrow 0\ \text{as}\ t\rightarrow\infty$$
with respect to the norm defined by $\|f\|=\sup\limits_{t\in [0,+\infty)}\|f(t)\|$.
\end{defn}

%There exists a local equivariant higher index map $$Ind_{\Gamma,L}:K_*^\Gamma(X)\rightarrow K_*(C_L^*(X)^\Gamma),$$ see Definition 3.1 in \cite{Yu97} or Definition 12 in \cite{Shan} for details.

Let us define the associated local $\Gamma$-index map. For each positive integer $n$, let $\{U_{n,i}\}_{j}$ be a locally finite and $\Gamma$-invariant open cover for $X$ such that diameter $(U_{n,i})< \frac{1}{n}$ for all $i$. Let $\{\phi_{n,i}\}_{j}$ be a continuous partition of unity subordinate to $\{U_{n,i}\}_{j}$. Let $[(H, \phi, F)]\in K_{0}^{\Gamma}(X)$. Define a family of operators $F(t)(t\in[0, \infty))$ acting on $H$ by
$$
F(t)=\sum_{i}\bigl((1-(t-n))\phi_{n,i}^{\frac{1}{2}}F\phi_{n,i}^{\frac{1}{2}}+(t-n)\phi_{n+1,i}^{\frac{1}{2}}F\phi_{n+1,i}^{\frac{1}{2}}\bigr)
$$
for all $t\in[n,n+1]$, where the infinite sum converges in strong topology. Notice that the propagation of $(F(t))\rightarrow 0$ as $t\rightarrow\infty.$
Using this it is no difficult to see that $F(t)$ is a multiplier of $C^*_L(X)^\Gamma$ and $F(t)$ is a unitary modulo of $C^*_L(X)^\Gamma$. Hence $F(t)$ gives rise to an element $\partial[F(t)]$ in $K_0(C^*_L(X)^\Gamma)$. So we  define the local index map $Ind_{\Gamma,L}$ from $K_0^\Gamma(X)$ to $K_0(C^*_L(X)^\Gamma)$. Similarly, we can define the local index map from $K_1^\Gamma(X)$ to $K_1(C^*_L(X)^\Gamma)$.

Now we have the following theorem which is an equivariant version of Theorem 3.2 in~\cite{Yu97} or  Theorem 3.4 in~\cite{Qiao}.

\begin{thm}\label{bcc}
If $X$ is a discrete metric space with bounded geometry and $\Gamma$
is a discrete group acting on $X$ properly and isometrically. Then for any $d>0$,  $$Ind_{\Gamma, L}:K_*^{\Gamma}(P_d(X))\rightarrow K_*(C_L^*(P_d(X))^{\Gamma})$$ is an isomorphism, where $P_d(X)$ is the Rips complex.
\end{thm}

For any $d>0$, let $e$ be the evaluation map from $C_L^*(P_d(X))^{\Gamma}$ to $C^*(P_d(X))^{\Gamma}$ by $e(f)=f(0)$ for $f\in C_L^*(P_d(X))^{\Gamma}$,  we have the following commutative diagram:
$$
\xymatrix{
                &         \lim\limits_{d\rightarrow \infty}K_*(C_L^*(P_d(X))^{\Gamma}) \ar[d]^{e_*}     \\
  \lim\limits_{d\rightarrow \infty}K_*^{\Gamma}(P_d(X)) \ar[ur]^{Ind_{\Gamma, L\ \ \ \ }} \ar[r]^{Ind_{\Gamma\ \ \ \ }} & \lim\limits_{d\rightarrow \infty}K_*(C^*(P_d(X))^{\Gamma}).             }
$$
So in order to prove that $Ind_\Gamma$ is an isomorphism or an injectivity, it suffices to prove that $$e_*:\lim\limits_{d\rightarrow \infty}K_*(C_L^*(P_d(X))^{\Gamma})\rightarrow \lim\limits_{d\rightarrow \infty}K_*(C^*(P_d(X))^{\Gamma})$$ is an isomorphism or an injectivity.

\section{Equivariant $\pi$-localization algebra}

In this section, we shall define an equivariant $\pi$-localization algebra for a $\Gamma$-space $X$ with bounded geometry and we prove that if the quotient space $X/\Gamma$ admits a coarse embedding into a Hilbert space, then there is an injective homomorphism from the $K$-theory of the equivariant $\pi$-localization algebra to the $K$-theory of the equivariant Roe algebra for the $\Gamma$-space $X$.

Let $\Gamma$ be a countable discrete group, and let $X$ be a $\Gamma$-space with bounded geometry. Assume that $X/\Gamma$ is coarse embeddable into  Hilbert space $H$, and $\xi$ is the coarse embedding map from $X/\Gamma$ into $H$. Let $\pi:X\rightarrow X/\Gamma$ be the quotient map.

\begin{defn}
Let $\mathbb{C}_{\pi,L}^*(X)^\Gamma$ be the algebra of all bounded and uniformly continuous functions  $$f:[0,+\infty)\rightarrow C^*(X)^\Gamma$$ such that  $$\sup\{d(\pi(x),\pi(y)):(x,y)\in \supp(f(t))\}\rightarrow 0 \ \text{as}\  t\rightarrow \infty.$$ We define $C^*_{\pi,L}(X)^\Gamma$ to be the norm closure of $\mathbb{C}_{\pi,L}^*(X)^\Gamma$ with respect to the norm defined by $\|f\|=\sup\limits_{t\in [0,+\infty)}\|f(t)\|$.

\end{defn}

Now there is an evaluation map $e_\pi: C^*_{\pi,L}(X)^\Gamma\rightarrow C^*(X)^\Gamma$, which induces a homomorphism $(e_{\pi})_*:K_*(C^*_{\pi,L}(X)^\Gamma)\rightarrow K_*(C^*(X)^\Gamma)$. We will show that $$(e_{\pi})_*:\lim\limits_{d\rightarrow \infty}K_*(C_{\pi,L}^*(P_d(X))^{\Gamma})\rightarrow \lim\limits_{d\rightarrow \infty}K_*(C^*(P_d(X))^{\Gamma})$$ is injective. To do this, we need some preparations.

\begin{defn}
Let $X$ and $Y$ be two $\Gamma$-spaces, let $f,g:X\rightarrow Y$ be $\Gamma$-equivariant coarse maps. We say that $f$ and $g$ are $\pi$-$strong\ Lipschitz\ homotopic$ and denoted by $f\sim_\pi g$ if there is a map $F:X\times [0,1]\rightarrow Y$ such that
\begin{enumerate}
\item[(1)] $F(t,x)$ is an equivariant coarse map from $X$ to $Y$ for each $t$;
\item[(2)] $d(\pi(F(t,x)),\pi(F(t,y)))\leq Cd(\pi(x),\pi(y))$ for all $x,y\in X$ and $t\in [0,1]$, where $C$ is a constant;
\item[(3)] for any $\epsilon>0$, there exists $\delta>0$ such that $d(F(t_1,x),F(t_2,x))<\epsilon$ for all $x\in X$ if $|t_1-t_2|<\delta$;
\item[(4)] $F(0,x)=f(x), F(1,x)=g(x)$ for all $x\in X$.
\end{enumerate}

\begin{defn}
Let $X$ and $Y$ be two $\Gamma$-spaces. We say that $X$ and $Y$ are $\pi$-$strong\ Lipchitz\ homotopy\ equivalent$ if there exist $\Gamma$-equivariant coarse maps $f:X\rightarrow Y$ and $g:Y\rightarrow X$ such that $g\circ f\sim_\pi id_X$ and $f\circ g\sim_\pi id_Y$.

\end{defn}

\end{defn}
\begin{prop}\label{sec:3:prop:3.4}
Let $X$ and $Y$ be two $\Gamma$-spaces. If $X$ and $Y$ are $\pi$-strong Lipschitz homotopy equivalent, then $K_*(C^*_{\pi,L}(X)^\Gamma)$ is isomorphic to $K_*(C^*_{\pi,L}(Y)^\Gamma)$.
\end{prop}

\begin{proof}
Let $F$ be the $\pi$-strong Lipschitz homotopy between $g\circ f$ and $id_X$ such that $F(0,x)=(g\circ f)(x)$ and $F(1,x)=x$.
Since $\Gamma$ acts on $X$ properly and isometrically, for any $\varepsilon>$0, we can find a sequence of disjoint Borel cover $\{E_i\}$ of $X$ such that for each $i$, (1) $E_i=\Gamma\times_{F_i}K_i$; (2) $K_i$ has non-empty interior; (3) the diameter of $K_i$ is no more than $\frac{\varepsilon}{2}$. Now we can define an equivariant isometry $V_{g\circ f}$ from $H_X$ to $H_X$ such that $V_{g\circ f}$ maps $\chi_{(g\circ f)^{-1}(K_i)}H_X$ to $\chi_{K_i}H_X$ with
\begin{enumerate}
  \item [(1)] $d((g\circ f)(x),y)$ is bounded for $(x,y)\in \supp(V_{g\circ f})$, this is because both $f$ and $g$ are coarse maps and the diameter of $K_i$ is no more than $\frac{\varepsilon}{2}$;
  \item [(2)] $d(\pi((g\circ f)(x)),\pi(y))\leq \varepsilon$ for $(x,y)\in \supp(V_{g\circ f})$.
\end{enumerate}

Let $\{\varepsilon_k\}_k$ be a sequence of positive numbers such that $\varepsilon_k\rightarrow 0$ as $k\rightarrow +\infty$. For each $k$, by using the above discussion, we can define an isometry $V_k$ such that $d((g\circ f)(x),y)$ is bounded for $(x,y)\in \supp(V_k)$ and $d(\pi((g\circ f)(x)),\pi(y))\leq \varepsilon_k$ for $(x,y)\in \supp(V_k)$.

The sequence of isometries $\{V_k\}_k$ induces a homomorphism $Ad_*(V_{g\circ f})$ from $K_*(C^*_{\pi,L}(X)^\Gamma)$ to $K_*(C^*_{\pi,L}(X)^\Gamma)$, and it is an identity homomorphism by using the $\pi$-strong Lipschitz homotopy equivalent between $X$ and $Y$(cf.~Proposition 3.7 in \cite{Yu97}). Similarly, we can define a homomorphism $Ad_*(V_{f\circ g})$ from $K_*(C^*_{\pi,L}(Y)^\Gamma)$ to $K_*(C^*_{\pi,L}(Y)^\Gamma)$ which is an identity homomorphism. So $K_*(C^*_{\pi,L}(X)^\Gamma)$ is isomorphic to $K_*(C^*_{\pi,L}(Y)^\Gamma)$.
\end{proof}

By using the $\pi$-strong Lipschitz homotopy invariance above, we have the following proposition which can be proved as Proposition 3.11 in~\cite{Yu97}.
\begin{prop}\label{sec:3;prop:3.5}
Let $X$ be a simplicial complex with proper and isometrical $\Gamma$ action
and endowed with the spherical metric. If $X_1$ and $X_2$ are two subcomplex of $X$ with the subspace metric, then we have the following six term exact sequence:
$$\xymatrix{
  K_0(A) \ar[d] \ar[r] & K_0(B_1)\oplus K_0(B_2) \ar[r] & K_0(C) \ar[d] \\
  K_1(C) \ar[r] &  K_1(B_1)\oplus K_1(B_2) \ar[r] & K_1(A),}$$
where $A=C_{\pi,L}^*(X_1\cap X_2)^\Gamma$, $C=C_{\pi,L}^*(X_1\cup X_2)^\Gamma$, $B_1=C_{\pi,L}^*(X_1)^\Gamma$, $B_2=C_{\pi,L}^*(X_2)^\Gamma$.
\end{prop}

Now we recall an algebra
associated to an infinite-dimensional Euclidean space introduced by Higson, Kasparov and Trout
\cite{HKT}. Let $H$ be a real (countably infinite dimensional) Hilbert space.  Denote by $V_a$, $V_b$ etc., the finite dimensional affine subspaces of $H$. Let $V_a^0$
be the finite dimensional linear subspaces of $H$ consisting of differences of elements
of $V_a$. Let Clifford$(V_a^0)$ be the complexified Clifford algebra of $V_a^0$, and $\mathcal {C}(V_a)$ be the graded $C^*$-algebra of continuous functions vanishing at infinity from $V_a$ to
Clifford$(V_a^0)$. Let $\mathcal {S}$ be the $C^*$-algebra $C_0(\mathbb{R})$, graded according to
the odd and even functions. Define the graded tensor product $$\mathcal {A}(V_a)=\mathcal {S}\hat{\otimes}\mathcal {C}(V_a).$$
If $V_a\subseteq V_b$, we have a decomposition $V_b=V_{ba}^0+V_a$, where $V_{ba}^0$ is the orthogonal complement of $V_a^0$ in $V_b^0$. For each $v_b \in V_b$, we have a corresponding decomposition $v_b=v_{ba}+v_a$, where $v_{ba}\in V_{ba}^0$ and $v_a\in V_a$. Every function $h$ on $V_a$ can be extended to a function $\tilde{h}$ on $V_b$ by the formula $\tilde{h}(v_{ba}+v_a)=h(v_a)$.

\begin{defn}\label{clif-alg}
If $V_a\subseteq V_b$, we use $C_{ba}$  to denote the function $V_b\rightarrow \textrm{Clifford}(V_b^0)$, $v_b\mapsto v_{ba}$, where $v_{ba}$ is
considered as an element of $\textrm{Clifford}(V_b^0)$ via the inclusion $V_{ba}^0\subset \textrm{Clifford}(V_b^0)$.
Let $X$ be the unbounded multiplier of $\mathcal{S}$  with degree one given by the function $x\mapsto x$.
Define a $*$-homomorphism $\beta_{ba}:\mathcal {A}(V_a)\rightarrow \mathcal {A}(V_b)$ by the formula $$\beta_{ba}(g\hat{\otimes}h)=g(X\hat{\otimes}1+1\hat{\otimes}C_{ba})(1\hat{\otimes}\tilde{h}),$$for $g\in \mathcal {S}$ and $h\in \mathcal {C}(V_a)$, where $g(X\hat{\otimes}1+1\hat{\otimes}C_{ba})$ is defined by functional calculus.
\end{defn}

\begin{rem}
Let $g_0(x)=e^{-x^2}$ and $g_1(x)=xe^{-x^2}$. It is not difficult to check that $$g_0(X\hat{\otimes}1+1\hat{\otimes}C_{ba})=g_0(X)\hat{\otimes}g_0(C_{ba})$$ for $g_0=e^{-x^2}$ and $$g_1(X\hat{\otimes}1+1\hat{\otimes}C_{ba})=g_0(X)\hat{\otimes}g_1(C_{ba})+g_1(X)\hat{\otimes}g_0(C_{ba})$$ for $g_1(x)=xe^{-x^2}$.
Moreover $g_0(X)=g_0(x)$, $g_1(X)=g_1(x)$ as multiplication operators, and $(g_0(C_{ba})\tilde{h})(v_b)=g_0(\|v_{ba}\|)h(v_a)$, $(g_1(C_{ba})\tilde{h})(v_b)=v_{ba}g_0(\|v_{ba}\|)h(v_a)$ when we view $C_{ba}$ as unbounded multiplier on $\mathcal {C}(V_b)$.
\end{rem}

The maps in Definition \ref{clif-alg} make the collection $(\mathcal {A}(V_a))$ into a directed system as $V_a$ ranges over finite dimensional affine subspaces of $H$. Define the $C^*$-algebra $\mathcal {A}(H,\xi)$ by $$\mathcal {A}(H,\xi)=\lim\limits_{\rightarrow}\mathcal {A}(V_a).$$

Now, let $\mathbb{R}_+\times H$ be endowed with the weakest topology for which the projection to $H$ is weakly continuous and the function $(t,w)\mapsto t^2+\|w\|^2$ is continuous. This topology makes $\mathbb{R}_+\times H$ into a locally compact Hausdorff space. Note that for $v\in H$ and $r>0$, $B(v,r)=\{(t,w)\in \mathbb{R}_+\times H: t^2+\|v-w\|^2<r^2\}$ is an open subset of $\mathbb{R}_+\times H$.
For each finite dimensional $V_a\subseteq H$, $C_0(\mathbb{R}_+ \times V_a)$ is included in $\mathcal {A}(V_a)$ as its center. If $V_a\subseteq V_b$, then $\beta_{ba}$ takes $C_0(\mathbb{R}_+\times V_a)$ into $C_0(\mathbb{R}_+\times V_b)$. Then $C^*$-algebra $\lim\limits_{\rightarrow}C_0(\mathbb{R}_+\times V_a)$ is isomorphic to $C_0(\mathbb{R}_+\times H)$,
where the direct limit is over the directed set of all finite dimensional affine subspaces $V_a$ of $H$.

\begin{defn}\label{supp_a}
The {\it support} of an element $a\in \mathcal {A}(H,\xi)$ is the complement of all $(t,v)\in \mathbb{R}_+\times H$ such that there exists $g\in C_0(\mathbb{R}_+\times H)$ with $g(t,v)\neq 0$ and $g\cdot a=0$.

\end{defn}

For each $d>0$, the proper and isometrical action on $X$ induces a proper and isometrical action on $P_d(X)$ by $\gamma \cdot \sum\limits_{i=1}^k c_i x_i=\sum\limits_{i=1}^k c_i( \gamma\cdot x_i)$ for $x_i\in X, i=1,2,\cdots, k$ and $\sum\limits_{i=1}^k c_i=1$. Then the coarse embedding $\xi:X/\Gamma\rightarrow H$ induces a coarse embedding $\xi:P_d(X)/\Gamma\rightarrow H$ by $\xi\Big(\pi(\sum\limits_{i=1}^k c_i x_i)\Big)=\sum\limits_{i=1}^k c_i \xi(\pi(x_i))$.
Take any element $\pi(x)\in P_d(X)/\Gamma$, we define
$$W_k(\pi(x))=\xi(\pi(x))+\,\mathrm{span}\{\xi(\pi(y))-\xi(\pi(x)) : y\in P_d(X)/\Gamma, d(\pi(x),\pi(y))\leq k^2\},$$
which is a finite dimensional affine subspace of $H$ by using the bounded geometry of $X$. We can assume that $\bigcup\limits_{k\geq 1} W_k(\pi(x))$ is dense in $H$, otherwise we can take $H$ to be the norm closure of $\bigcup\limits_{k\geq 1} W_k(\pi(x))$.
For $k<k'$, let $\beta_{k',k}(\pi(x)):\mathcal {A}(W_k(\pi(x)))\rightarrow \mathcal {A}(W_{k'}(\pi(x)))$ be defined as in Definition \ref{clif-alg} and $\beta_k(\pi(x)):\mathcal {A}(W_k(\pi(x))\rightarrow \mathcal {A}(H,\xi)$ coming from the definition of $\mathcal {A}(H,\xi)$. We write $\beta(\pi(x))$ for $$\beta_0(\pi(x)):\mathcal {S}\cong \mathcal {A}(W_0(\pi(x)))\rightarrow \mathcal {A}(H,\xi).$$

Choose a countable $\Gamma$-invariant dense subset $X_d$ of $P_d(X)$ for each $d>0$ such that $X_{d_1}\subseteq X_{d_2}$ if $d_1\leq d_2$. Let $C_{alg}^*(P_d(X), \mathcal {A}(H,\xi))^\Gamma$ be the set of all functions $T$ on $X_d\times X_d$ such that
\begin{enumerate}
\item[(1)] there exists an integer $N$ such that $$T(x,y)\in (\beta_N(\pi(x))\hat{\otimes}1)(\mathcal {A}(W_N(\pi(x)))\hat{\otimes}K)\subseteq \mathcal {A}(H,\xi)\hat{\otimes}K$$ for all $x,y\in X_d$, where $\beta_N(\pi(x)):\mathcal {A}(W_N(\pi(x)))\rightarrow \mathcal {A}(H,\xi)$ is the $*$-homomorphism associated to the inclusion of $W_N(\pi(x))$ into $H$, and $K$ is the algebra of compact operators;
\item[(2)] there exists $M\geq 0$ such that $\|T(x,y)\|\leq M$ for all $x,y\in X_d$;
\item[(3)] there exists $L>0$ such that for each $y\in X_d$, $$\sharp\{x:T(x,y)\neq 0\}\leq L,\qquad \sharp\{x:T(y,x)\neq 0\}\leq L;$$
\item[(4)] there exists $r_1>0$  such that if $d(x,y)>r_1$, then $T(x,y)=0$;
\item[(5)] there exists $r_2>0$ such that
\begin{align*}
\supp(T(x,y))&\subseteq B_{\mathbb{R}_+\times H }(\xi(\pi(x)),r_2)\\
&:=\left\{(s,h)\in\mathbb{R}_+\times H:s^2+\|h-\xi(\pi(x))\|^2<r_2^2\right\}
\end{align*}
for $x,y\in X_d$;
\item[(6)] $\gamma(T)=T$, where $\gamma(T)(x,y)=T(\gamma^{-1}x,\gamma^{-1}y)$.
\end{enumerate}

If we define a product structure on $C^*_{alg}(P_d(X),\mathcal {A}(H,\xi))^\Gamma$ by $$(T_1 T_2)(x,y)=\sum\limits_{z\in X_d}T_1(x,z)T_2(z,y),$$
then $C^*_{alg}(P_d(X),\mathcal {A}(H,\xi))^\Gamma$ is made into a $*$-algebra via matrix multiplication together with the $*$-operation on $\mathcal{A}(H,\xi)\hat{\otimes}K$.
Let $$E=\left\{\sum\limits_{x\in X_d}a_x[x]:a_x\in \mathcal {A}(H,\xi)\hat{\otimes} K,\sum\limits_{x\in X_d}a_x^*a_x\ \text{converges in norm}\right\}.$$
We know that $E$ is a Hilbert module over $\mathcal {A}(H,\xi)\hat{\otimes} K$ with
$$\left\langle\sum\limits_{x\in X_d}a_x[x],\sum\limits_{x\in X_d}b_x[x]\right\rangle=\sum\limits_{x\in X_d}a_x^*b_x,$$
$$\left(\sum\limits_{x\in X_d}a_x[x]\right)a=\sum\limits_{x\in X_d}a_xa[x]$$
for all $a\in \mathcal {A}(H,\xi)\hat{\otimes} K$ and $\sum\limits_{x\in X_d}a_x[x]\in E$.

$C^*_{alg}(P_d(X),\mathcal {A}(H,\xi))^{\Gamma}$ acts on $E$ by
$$T\left(\sum\limits_{x\in X_d}a_x[x]\right)=\sum\limits_{y\in X_d}\left(\sum\limits_{x\in X_d}T(y,x)a_x\right)[y],$$
where $T\in C^*_{alg}(P_d(X),\mathcal {A}(H,\xi))^{\Gamma}$ and $\sum\limits_{x\in X_d}a_x[x]\in E$.
\begin{defn}
The {\it twisted equivariant Roe algebra} $C^*(P_d(X),\mathcal {A}(H,\xi))^{\Gamma}$ is defined to be the operator norm closure of $C^*_{alg}(P_d(X),\mathcal {A}(H,\xi))^{\Gamma}$ in $\mathcal {B}(E)$, the $C^*$-algebra of all module homomorphisms from $E$ to $E$ for which there is an adjoint module homomorphism.
\end{defn}

Let $C_{\pi,L,alg}^*(P_d(X),\mathcal {A}(H,\xi))^{\Gamma}$ be the set of all bounded, uniformly norm continuous functions
$$g:\mathbb{R}_+\rightarrow C_{alg}^*(P_d(X),\mathcal {A}(H,\xi))^\Gamma$$
such that:
\begin{enumerate}
  \item[(1)] there exists $N$ such that $g(t)(x,y)\in (\beta_N(\pi(x))\hat{\otimes}1)(\mathcal {A}(W_N(\pi(x)))\hat{\otimes}K)\subseteq \mathcal {A}(H,\xi)\hat{\otimes}K$ for all $t\in \mathbb{R}_+$, $x,y\in X_d$;
  \item[(2)] $\sup\{d(\pi(x),\pi(y)):g(t)(x,y)\neq 0\}\rightarrow 0 \ \text{as}\  t\rightarrow \infty$;
  \item[(3)] there exists $R>0$ such that $\supp(g(t)(x,y))\subseteq B_{\mathbb{R}_+\times H}(\xi(\pi(x)),R)$ for all $t\in \mathbb{R}_+, x,y\in X_d$;
  \item[(4)] there exists $L>0$ such that for any $y\in P_d(X)$, $\sharp\{x:g(t)(x,y)\neq 0\}<L$, $\sharp\{x:g(t)(y,x)\neq 0\}<L$ for any $t\in \mathbb{R}_+$;
\end{enumerate}

\begin{defn}\label{local-alg}
The {\it twisted equivariant $\pi$-localization algebra} $C_{\pi,L}^*(P_d(X),\mathcal {A}(H,\xi))^{\Gamma}$ is the completion of $C_{\pi,L,alg}^*(P_d(X),\mathcal {A}(H,\xi))^{\Gamma}$ with respect to the norm
$$\|g\|=\sup\limits_{t\in \mathbb{R}_+}\|g(t)\|_{C^*(P_d(X),\mathcal {A}(H,\xi))^\Gamma}.$$
\end{defn}

With the similar way, we can define the twisted equivariant localization algebra $C_{L}^*(P_d(X),\mathcal {A}(H,\xi))^{\Gamma}$ only by changing $\sup\{d(\pi(x),\pi(y)):g(t)(x,y)\neq 0\}\rightarrow 0 \ \text{as}\  t\rightarrow \infty$ to $\sup\{d(x,y):g(t)(x,y)\neq 0\}\rightarrow 0 \ \text{as}\  t\rightarrow \infty$.

We have an evaluation homomorphism
$$e'_\pi:C_{\pi,L}^*(P_d(X),\mathcal {A}(H,\xi))^{\Gamma}\rightarrow C^*(P_d(X),\mathcal {A}(H,\xi))^{\Gamma}$$ defined by $e'_\pi(g)=g(0)$. This map induces a homomorphism at $K$-theory level:
$$(e'_\pi)_*:\lim\limits_{d\rightarrow \infty}K_*(C_{\pi,L}^*(P_d(X),\mathcal {A}(H,\xi))^{\Gamma})\rightarrow \lim\limits_{d\rightarrow \infty}K_*(C^*(P_d(X),\mathcal {A}(H,\xi))^{\Gamma}).$$
We will show that $(e'_\pi)_*$ is an isomorphism. To prove this, we need the following definitions.

\begin{defn}
(1) The {\it support} of an element $T$ in $C_{alg}^*(P_d(X),\mathcal {A}(H,\xi))^{\Gamma}$ is defined to be
$$\left\{(x,y,t,v)\in X_d\times X_d\times \mathbb{R}_+\times H:(t,v)\in \supp(T(x,y))\right\}.$$

(2) The {\it support} of an element $g$ in $C_{L,alg}^*(P_d(X)_,\mathcal {A}(H,\xi))^{\Gamma}$ is defined to be
$$\bigcup\limits_{t\in \mathbb{R}_+}\supp(g(t)).$$

(3)  The {\it support} of an element $g$ in $C_{\pi,L,alg}^*(P_d(X)_,\mathcal {A}(H,\xi))^{\Gamma}$ is defined to be
$$\bigcup\limits_{t\in \mathbb{R}_+}\supp(g(t)).$$

\end{defn}

Let $O$ be an open subset of $\mathbb{R}_+\times H$. Define $C_{alg}^*(P_d(X),\mathcal {A}(H,\xi))^{\Gamma}_O$ to be the subalgebra of $C_{alg}^*(P_d(X),\mathcal {A}(H,\xi))^{\Gamma}$ consisting of all elements whose supports are contained in $X_d\times X_d\times O$, i.e.,
\begin{align*}
&C_{alg}^*(P_d(X),\mathcal {A}(H,\xi))^{\Gamma}_O\\
&\hspace{30mm} =\left\{T\in C_{alg}^*(P_d(X),\mathcal {A}(H,\xi))^{\Gamma}:\supp(T(x,y))\subseteq O, \forall x,y\in X_d\right\}.
\end{align*}
Define $C^*(P_d(X),\mathcal {A}(H,\xi))^{\Gamma}_O$ to be the norm closure of $C_{alg}^*(P_d(X),\mathcal {A}(H,\xi))^{\Gamma}_O$. Similarly, let
\begin{align*}
&C_{L,alg}^*(P_d(X),\mathcal {A}(H,\xi))^{\Gamma}_O\\
&\hspace{30mm} =\left\{g\in C_{L,alg}^*(P_d(X),\mathcal {A}(H,\xi))^{\Gamma}:\supp(g)\subseteq X_d\times X_d\times O\right\},
\end{align*}
\begin{align*}
&C_{\pi,L,alg}^*(P_d(X),\mathcal {A}(H,\xi))^{\Gamma}_O\\
&\hspace{30mm} =\left\{g\in C_{\pi,L,alg}^*(P_d(X),\mathcal {A}(H,\xi))^{\Gamma}:\supp(g)\subseteq X_d\times X_d\times O\right\}.
\end{align*}
We define that
$C_{L}^*(P_d(X),\mathcal {A}(H,\xi))^{\Gamma}_O$ and
$C_{\pi,L}^*(P_d(X),\mathcal {A}(H,\xi))^{\Gamma}_O$ are the norm closures of $C_{L,alg}^*(P_d(X),\mathcal {A}(H,\xi))^{\Gamma}_O$ and
$C_{\pi,L,alg}^*(P_d(X),\mathcal {A}(H,\xi))^{\Gamma}_O$ respectively.

\begin{prop}\label{sec:3:thm:3.11}
Let $\Gamma$ be a countable discrete group, $X$ a $\Gamma$-space, $H$ a real Hilbert space. If $X/\Gamma$ admits an coarse embedding into $H$, then the homomorphism
$$(e'_\pi)_*:\lim\limits_{d\rightarrow \infty}K_*(C_{\pi,L}^*(P_d(X),\mathcal {A}(H,\xi))^{\Gamma})\rightarrow \lim\limits_{d\rightarrow \infty}K_*(C^*(P_d(X),\mathcal {A}(H,\xi))^{\Gamma})$$ is an isomorphism.
\end{prop}

\begin{proof}
For any $r>0$, we define $O_r\subseteq \mathbb{R}_+\times H$ by $$O_r=\bigcup\limits_{\pi(x)\in X/\Gamma}B(\xi(\pi(x)),r),$$
where $\xi$ is the coarse embedding from $X/\Gamma$ to $H$.
We have $$C^*_{\pi,L}(P_d(X),\mathcal {A}(H,\xi))^\Gamma=\lim\limits_{r\rightarrow \infty}C^*_{\pi,L}(P_d(X),\mathcal {A}(H,\xi))_{O_r}^\Gamma,$$
$$C^*(P_d(X),\mathcal {A}(H,\xi))^\Gamma=\lim\limits_{r\rightarrow \infty}C^*(P_d(X),\mathcal {A}(H,\xi))_{O_r}^\Gamma.$$
It is not difficult to check that (cf.~Theorem 3.3 in \cite{ShanWang})
$$\lim\limits_{d\rightarrow \infty}\lim\limits_{r\rightarrow \infty}K_*(C^*(P_d(X),\mathcal {A}(H,\xi))_{O_r}^\Gamma)\cong \lim\limits_{r\rightarrow \infty}\lim\limits_{d\rightarrow \infty}K_*(C^*(P_d(X),\mathcal {A}(H,\xi))_{O_r}^\Gamma)$$
$$\lim\limits_{d\rightarrow \infty}\lim\limits_{r\rightarrow \infty}K_*(C_{\pi,L}^*(P_d(X),\mathcal {A}(H,\xi))_{O_r}^\Gamma)\cong \lim\limits_{r\rightarrow \infty}\lim\limits_{d\rightarrow \infty}K_*(C_{\pi,L}^*(P_d(X),\mathcal {A}(H,\xi))_{O_r}^\Gamma)$$
and we can get the following commuting diagram:
$$\xymatrix{
 \lim\limits_{d\rightarrow\infty}K_*(C^{*}_{\pi,L}(P_{d}(X), \mathcal {A}(H,\xi))^\Gamma) \ar[d]_{\cong} \ar[r]^{(e'_\pi)_*} & \lim \limits_{d\rightarrow\infty}K_*(C^{*}(P_{d}(X), \mathcal {A}(H,\xi))^\Gamma) \ar[d]_{\cong}  \\
 \lim \limits_{d\rightarrow\infty}\lim \limits_{r\rightarrow\infty}K_*(C_{\pi,L}^*(P_d(X),\mathcal {A}(H,\xi))_{O_r}^\Gamma) \ar[d]_{\cong}  \ar[r]^{(e'_\pi)_*} & \lim \limits_{d\rightarrow\infty}\lim \limits_{r\rightarrow\infty}K_*(C^*(P_d(X),\mathcal {A}(H,\xi))_{O_r}^\Gamma)  \ar[d]_{\cong}  \\
   \lim \limits_{r\rightarrow\infty}\lim \limits_{d\rightarrow\infty}K_*(C_{\pi,L}^*(P_d(X),\mathcal {A}(H,\xi))_{O_r}^\Gamma) \ar[r]^{(e'_\pi)_*} & \lim \limits_{r\rightarrow\infty}\lim \limits_{d\rightarrow\infty}K_*(C^*(P_d(X),\mathcal {A}(H,\xi))_{O_r}^\Gamma).}$$
So to prove Proposition \ref{sec:3:thm:3.11}, it suffices to show that for any $r_0>0$,
\begin{align*}
(e'_\pi)_*:&\lim \limits_{d\rightarrow\infty}K_*(\lim\limits_{r<r_0,r\rightarrow r_0}C_{\pi,L}^*(P_d(X),\mathcal {A}(H,\xi))_{O_r}^\Gamma) \\
&\hspace{50mm} \rightarrow\lim \limits_{d\rightarrow\infty}K_*(\lim\limits_{r<r_0,r\rightarrow r_0}C^*(P_d(X),\mathcal {A}(H,\xi))_{O_r}^\Gamma)
\end{align*}
is an isomorphism.

Let $r_0>0$. Since $X$ has bounded geometry and $\Gamma$ acts on $X$ properly and isometrically, then $\widetilde{X}=X/\Gamma$ is a bounded geometrical metric space. So there exists finitely many mutually disjoint subsets of $\widetilde{X}$, say $\widetilde{X}_k:=\{\pi(x_j):j\in J_k\}$ with some index set $J_k$ for $k=1,2,\cdots, k_0$, such that $\widetilde{X}=\bigsqcup\limits_{k=1}^{k_0}\widetilde{X_k}$ and for $\pi(x_i),\pi(x_j)\in \widetilde{X}_k$, $d(\pi(x_i),\pi(x_j))$ is large enough such that $d(\xi(\pi(x_i)), \xi(\pi(x_j)))> 2r_0$.

For each $k=1,2,\cdots, k_0$ and $0<r<r_0$, we denote
$$O_{r,k}=\bigcup\limits_{j\in J_k}\big(B(\xi(\pi(x_j)), r)\big),$$
we will complete the proof of this theorem by using the Mayer-Vietoris sequence argument as Theorem 6.8 in~\cite{Yu} once we proved that $$(e'_\pi)_*:\lim\limits_{d\rightarrow \infty}K_*(C_{\pi,L}^*(P_d(X),\mathcal {A}(H,\xi))_{O_{r,k}}^\Gamma)\rightarrow\lim\limits_{d\rightarrow \infty}K_*(C^*(P_d(X),\mathcal {A}(H,\xi))_{O_{r,k}}^\Gamma)$$
is an isomorphism. For simplicity, we write $O_{r,k}$ to be $O$ and $B(\xi(\pi(x_j)), r)$ to be $O_j$.

Since the $\Gamma$-action on $\mathcal{A}(H,\xi)$ is trivial, for any $d>0$, we can show that (cf. Lemma 6.4 in~\cite{Yu} or Lemma 4.1 in~\cite{ShanWang})
$$C^*(P_d(X),\mathcal {A}(H,\xi))^\Gamma_O\cong \prod\limits_{j=1}^\infty C^*_r(\Gamma)\otimes \mathcal{A}(H,\xi)_{O_j}\otimes K(H)$$
where $\mathcal{A}(H,\xi)_{O_j}=\{a\in\mathcal{A}(H,\xi):\supp(a)\in O_j\}$.
Let $A_L^*$ be the algebra which is defined by the supremum norm closure of the set of uniformly continuous and uniformly bounded functions
$$g:[0,+\infty)\rightarrow \prod\limits_{j=1}^\infty C^*_r(\Gamma)\otimes \mathcal {A}(H,\xi)_{O_j}\otimes K(H).$$
By using the $\pi$-strong Lipschitz homotopy equivalent and a similar argument to the proof of Lemma 6.4 in~\cite{Yu}, we know that $\lim\limits_{d\rightarrow \infty}K_*(C_{\pi,L}^*(P_d(X),\mathcal {A}(H,\xi))_O^{\Gamma})$ is isomorphic to $K_*(A_L^*)$. Let $A_{L,0}^*$ be the ideal of $A_L^*$ with $g(0)=0$. Then by the Eilenberg swindle argument we have $$K_*(A_{L,0}^*)=0.$$
Moreover, $$0\rightarrow A_{L,0}^*\rightarrow A_L^*\rightarrow \prod\limits_{j=1}^\infty C^*_r(\Gamma)\otimes \mathcal {A}(H,\xi)_{O_j}\otimes K(H)$$ is exact,
so we can prove that $$K_*(A_L^*)\cong K_*(\prod\limits_{j=1}^\infty C^*_r(\Gamma)\otimes \mathcal {A}(H,\xi)_{O_j}\otimes K(H))$$ by the six-term exact sequence of $K$-theory, so $$\lim\limits_{d\rightarrow \infty}K_*(C_{\pi,L}^*(P_d(X),\mathcal {A}(H,\xi))_O^{\Gamma})\rightarrow \lim\limits_{d\rightarrow \infty}K_*(C^*(P_d(X),\mathcal {A}(H,\xi))_O^{\Gamma})$$ is isomorphic.
\end{proof}

\begin{thm}\label{sec:3:thm:3.12}
Let $\Gamma$ be a countable discrete group, $X$ a $\Gamma$-space, $H$ a real Hilbert space. If $X/\Gamma$ admits an coarse embedding into $H$, then the homomorphism $$(e_{\pi})_*:\lim\limits_{d\rightarrow \infty}K_*(C_{\pi,L}^*(P_d(X))^{\Gamma})\rightarrow \lim\limits_{d\rightarrow \infty}K_*(C^*(P_d(X))^{\Gamma})$$ is injective.
\end{thm}

\begin{proof}
For any $d>0$ and for each $t\geq 1$, define a map
$$\beta_t:\mathcal {S}\hat{\otimes}C_{alg}^*(P_d(X))^\Gamma\rightarrow C^*(P_d(X),\mathcal {A}(H,\xi))^\Gamma$$ by$$(\beta_t(g\hat{\otimes }T))(x,y)=(\beta(\pi(x)))(g_t)\hat{\otimes}T(x,y)$$ for all $g\in \mathcal {S}$ and $T\in C_{alg}^*(P_d(X))^\Gamma$, where $g_t(s)=g(t^{-1}s)$ for all $s\in \mathbb{R}$, and $\beta(\pi(x)):\mathcal {S}=\mathcal {A}(\xi(\pi(x)))\rightarrow \mathcal {A}(H,\xi)$ is the $*$-homomorphism from Definition \ref{clif-alg}.

Similarly, we can define $$(\beta_{\pi,L})_t:\mathcal {S}\hat{\otimes}C_{\pi,L,alg}^*(P_d(X))^\Gamma\rightarrow C_{\pi,L}^*(P_d(X),\mathcal {A}(H,\xi))^\Gamma.$$
The maps $\beta_t$ and $(\beta_{\pi,L})_t$ extend to asymptotic morphisms (cf. Lemma 7.6 in~\cite{Yu} or Lemma 5.8 in~\cite{Fu})
$$\beta: \mathcal {S}\hat{\otimes}C^*(P_d(X))^\Gamma \rightarrow C^*(P_d(X), \mathcal {A}(H,\xi))^\Gamma;$$
$$\beta_{\pi,L}: \mathcal {S}\hat{\otimes}C_{\pi,L}^*(P_d(X))^\Gamma \rightarrow C_{\pi,L}^*(P_d(X), \mathcal {A}(H,\xi))^\Gamma.$$
So they induce the homomorphisms $$(\beta)_*: K_*(C^*(P_d(X))^\Gamma) \rightarrow K_*(C^*(P_d(X), \mathcal {A}(H,\xi))^\Gamma);$$
$$(\beta_{\pi,L})_*: K_*(C_{\pi,L}^*(P_d(X))^\Gamma) \rightarrow K_*(C_{\pi,L}^*(P_d(X), \mathcal {A}(H,\xi))^\Gamma).$$
And we get a commutative diagram:
$$\xymatrix{
 \lim\limits_{d\rightarrow \infty}K_*(C_{\pi,L}^*(P_d(X))^{\Gamma})\ar[d]_{(\beta_{\pi,L})_*} \ar[r]^{(e_\pi)_*}  & \lim\limits_{d\rightarrow \infty}K_*(C^*(P_d(X))^{\Gamma}) \ar[d]_{(\beta)_*}  \\
  \lim\limits_{d\rightarrow \infty}K_*(C_{\pi,L}^*(P_d(X),\mathcal {A}(H,\xi))^{\Gamma})  \ar[r]^{(e'_\pi)_*} &    \lim\limits_{d\rightarrow \infty}K_*(C^*(P_d(X),\mathcal {A}(H,\xi))^{\Gamma}).}$$

By the induction on the dimension of skeletons of $P_d(X)$ and Proposition \ref{sec:3;prop:3.5}, together with the Bott Periodicity theorem, we know that $(\beta_{\pi,L})_*$ is an isomorphism. This, together with Proposition \ref{sec:3:thm:3.11} implies that
$(e_{\pi})_*$ is injective.
\end{proof}

\section{The $K$-theory of twisted equivariant localization algebra and $\pi$-localization algebra }

In this section, we shall prove that there is an injective map from the $K$-theory of the twisted equivariant localization algebra to the $K$-theory of the twisted equivariant $\pi$-localization algebra for the metric space $X$, where the twisted equivariant localization algebra and $\pi$-localization algebra are defined by using the coarse embedding from $\Gamma$ to the Hilbert space $H$. To do so, we shall prove that the $K$-theory of the twisted equivariant localization algebra is isomorphic to the $K$-theory of the twisted equivariant $\pi$-localization algebra for the metric space $X$, where the twisted equivariant localization algebra and $\pi$-localization algebra are defined by using the coarse embeddings from $\Gamma$ to the Hilbert space $H$ and from $X/\Gamma$ to $H$.

First, let us recall about the transformation groupoid and the relations between the coarse embedding of $\Gamma$ and the proper affine action of the transformation groupoid on a continuous field of Hilbert spaces. Let $\Gamma$ be a countable discrete group. Denote by $e$ its identity element. Assume $\Gamma$ acts on the right on a topological space $Z$. We use $Z\rtimes \Gamma$ to denote the transformation groupoid, where the product and the inverse operations of $Z\rtimes \Gamma$ are given by: $(z,\gamma)(z',\gamma')=(z, \gamma\gamma')$ for $(z,\gamma)$ and $(z', \gamma')\in Z\rtimes \Gamma$ satisfying $z'=z\gamma$, and $(z,\gamma)^{-1}=(z\gamma, \gamma^{-1})$ for $(z,\gamma)\in Z\rtimes \Gamma$.

Second, let us recall about the continuous field of Hilbert spaces (\cite{Dix}, pp. 211-212). $\big((H_z)_{z\in Z}, \Lambda\big)$ is called {\it a continuous field of a family $\{H_z\}_{z\in Z}$ of Hilbert spaces, with $\Lambda\subseteq \bigsqcup\limits_{z\in Z}H_z$} if the following conditions are true:
\begin{itemize}
  \item $\Lambda$ is a complex linear subspace of $\bigsqcup\limits_{z\in Z}H_z$;
  \item for every $z\in Z$, the set $\lambda(z)$ for $\lambda\in \Lambda$ is dense in $H_z$;
  \item for every $\lambda\in \Lambda$, the function $z\rightarrow \|\lambda(z)\|$ is continuous;
  \item let $\lambda\in \bigsqcup\limits_{z\in Z}H_z$, if for every $z\in Z$ and every $\varepsilon>0$, there exists a $\lambda'\in \Lambda$ such that $\|\lambda(y)-\lambda'(y)\|\leq \varepsilon$ throughout some neighborhood of $z$, then $\lambda\in \Lambda$.
\end{itemize}
We will abbreviate $\big((H_z)_{z\in Z}, \Lambda\big)$ to $(H_z)_{z\in Z}$ in this paper.

\begin{defn}\label{sec:4:def:3}
Let $\mathcal {H}=(H_z)_{z\in Z}$ be a continuous field of Hilbert spaces over $Z$. We say that the transformation groupoid $Z\rtimes \Gamma$ acts on $\mathcal {H}$ by {\it affine isometries} if for every $(z,\gamma)$, there is an affine isometry $\alpha_{(z,\gamma)}:H_{z\gamma}\rightarrow H_z$ such that
\begin{enumerate}
  \item [(1)] $\alpha_{(z,e)}:H_{z}\rightarrow H_z$ is the identity map;
  \item [(2)] $\alpha_{(z,\gamma)}\alpha_{(z',\gamma')}=\alpha_{(z,\gamma\gamma')}$ if $z'=z\gamma$;
  \item [(3)] for every continuous vector field $h(z)$ in $\mathcal {H}$ and every $\gamma\in\Gamma$, $\alpha_{(z,\gamma)}(h(z\gamma))$ is a continuous vector field in $\mathcal {H}$.
\end{enumerate}
\end{defn}

\begin{defn}[Tu, \cite{Tu}]
Let $Z\rtimes \Gamma$ act on $\mathcal {H}$ as in Definition \ref{sec:4:def:3}. The action is said to be {\it proper} if for any $R>0$ and $z\in Z$, the number of the elements in $\{\gamma\in\Gamma:\alpha_{(z,\gamma)}(B_{H_{z\gamma}}(R))\cap B_{H_z}(R)\neq \emptyset\}$ is finite, where $B_{H_z}(R)=\{v\in H_z: \|v\|\leq R\}$.
\end{defn}

In fact, by Proposition 6.5 and Lemma 6.7 in~\cite{Skan} and Proposition 3.8 in~\cite{Tu}, we have the following proposition:
\begin{prop}
The following are equivalent:
\begin{enumerate}
  \item [(1)] $\eta:\Gamma\rightarrow H$ is a coarse embedding;
  \item [(2)] there exists a convex, compact, Hausdorff, second countable space $Z$ with a right action of $\Gamma$ which admits a continuous, proper negative type function on $Z\rtimes \Gamma$ (see Definition 6.4 in~\cite{Skan});
  \item [(3)] $Z\rtimes \Gamma$ in (2) admits a proper action on a continuous field of affine Hilbert spaces $\mathcal {H}=(H_z)_{z\in Z}$.
\end{enumerate}
\end{prop}

For each Hilbert space $H_z$ of $\mathcal {H}=(H_z)_{z\in Z}$, let $\mathcal {A}(H_z)$ be the $C^*$-algebra which is defined below Definition \ref{clif-alg}. We use $\Theta(Z, \bigsqcup\limits_{z\in Z}\mathcal {A}(H_z))$ to denote the set of sections $f: Z\rightarrow\bigsqcup\limits_{z\in Z}\mathcal {A}(H_z)$ satisfying the following conditions:
\begin{enumerate}
  \item [(1)] $f(z)\in \mathcal {A}(H_z)$ for each $z\in Z$;
  \item [(2)] the continuous sections from $Z$ to $H_z$ is contained in $\Theta(Z, \bigsqcup\limits_{z\in Z}\mathcal {A}(H_z))$;
  \item [(3)] for any $z\in Z$ and $n\in \mathbb{N}$, there exists an open set $U_z\subseteq Z$ and continuous sections $a_0(z),a_1(z), \cdots, a_n(z):Z\rightarrow (H_z)_{z\in Z}$ such that

  \begin{itemize}
    \item $a_0(y), a_1(y), \cdots, a_n(y)$ are norm one and affine independent when $y\in U_z$;
    \item the isometry via the map
    $$C_b(U_z)\otimes \mathcal {A}(\mathbb{R}^n)\rightarrow\Theta(Z, \bigsqcup\limits_{z\in Z}\mathcal {A}(H_z))|_{U_z},\quad \mathbb{1}_{U_z}\otimes e_i\mapsto  a_i|_{U_z}$$
    induces {\it a continuous field structure} on $\Theta(Z, \bigsqcup\limits_{z\in Z}\mathcal {A}(H_z))$ (cf. \cite{HG}), where $\mathbb{1}_{U_z}$ is the constant function on $U_z$ with value $1$, $e_i(1\leq i\leq n)$ are the orthonormal basis of $\mathbb{R}^n$, $\mathcal {A}(\mathbb{R}^n)$ is defined as in section 3.
  \end{itemize}
\end{enumerate}
It is easy to check that the collection of $C^*$-algebras $(\mathcal{A}(H_z))_{z\in Z}$ with the sections $\Theta(Z, \bigsqcup\limits_{z\in Z}\mathcal {A}(H_z))$ is a continuous field of $C^*$-algebras (see Definition 10.3.1 in~\cite{Dix}).

\begin{defn}\label{proper alg}
For any section $f\in \Theta(Z, \bigsqcup\limits_{z\in Z}\mathcal {A}(H_z))$, put $$\|f\|=\sup\limits_{z\in Z}\|f(z)\|.$$
The $C^*$-algebra $\mathcal {A}(\mathcal {H}, \eta)$ is defined to be the completion of $\Theta(Z, \bigsqcup\limits_{z\in Z}\mathcal {A}(H_z))$ with the above norm.
\end{defn}

\begin{lem}
If $\Gamma$ is coarse embeddable into a Hilbert space, then $\mathcal {A}(\mathcal {H},\eta)$ above is a $\Gamma$-proper $C^*$-algebra.
\end{lem}

\begin{proof}
Let $\alpha$ be the proper affine isometric action of the groupoid $Z\rtimes \Gamma$ on the continuous Hilbert field $\mathcal {H}=(H_z)_{z\in Z}$.
For any element $f\in\mathcal {A}(\mathcal {H},\eta)$, define $$(\gamma(f)(z))(t,v)=f(z\gamma)(t, \alpha_{(z,\gamma)^{-1}}(v)),$$ where $z\in Z, t\in \mathbb{R}, v\in H_z$. This is an action of $\Gamma$ on $\mathcal {A}(\mathcal {H},\eta)$. Moreover, $\mathcal {A}(\mathcal {H},\eta)$ is a $\Gamma$-proper $C^*$-algebra since $\alpha$ is a proper affine isometric action on $\mathcal {H}=(H_z)_{z\in Z}$.
\end{proof}

Now let's define another version twisted Roe algebra by using the coarse embeddings of the quotient space $X/\Gamma$ and the group $\Gamma$. Let $\mathcal {A}(H,\xi)$ be the $C^*$-algebra which is defined in section 3 by using the coarse embeddings of the quotient space $X/\Gamma$, $\mathcal {A}(\mathcal{H},\eta)$ be the proper $C^*$-algebra defined in Definition \ref{proper alg}.

Choose a countable $\Gamma$-invariant dense subset $X_d$ of $P_d(X)$ for each $d>0$ such that $X_{d_1}\subseteq X_{d_2}$ if $d_1\leq d_2$. Let $C_{alg}^*(P_d(X),\mathcal {A}(H,\xi)\hat{\otimes}\mathcal {A}(\mathcal {H},\eta))^\Gamma$ be the set of all functions $T$ on $X_d\times X_d$ such that
\begin{enumerate}
\item[(1)] $T(x,y)\in\mathcal {A}(H,\xi)\hat{\otimes}\mathcal {A}(\mathcal {H},\eta)\hat{\otimes}K$ for $x,y\in X_d$;
\item[(2)] there exists $M\geq 0$ such that $\|T(x,y)\|\leq M$ for all $x,y\in X_d$;
\item[(3)] there exists $L>0$ such that for each $y\in X_d$, $$\sharp\{x:T(x,y)\neq 0\}\leq L,\qquad \sharp\{x:T(y,x)\neq 0\}\leq L;$$
\item[(4)] there exists $r_1>0$ and $r_2>0$ such that
\begin{itemize}
  \item $T(x,y)=0$ when $d(x,y)>r_1$;
  \item $T(x,y)$ is the finite sum of $T_1(x,y)\hat{\otimes}T_2(x,y)\hat{\otimes}k$ with $\text{Supp} (T_1(x,y))\subseteq B_{\mathbb{R}_+\times H }(\xi(\pi(x)),r_2)$, where $T_1(x,y)\in \mathcal {A}(H,\xi)$, $T_2(x,y)\in\mathcal {A}(\mathcal {H},\eta)$, $k\in K$;
\end{itemize}

\item[(5)] $\gamma(T)=T$, where $\gamma(T)(x,y)=\gamma(T(\gamma^{-1}x,\gamma^{-1}y))$.
\end{enumerate}

We define a product structure on $C^*_{alg}(P_d(X),\mathcal {A}(H,\xi)\hat{\otimes}\mathcal {A}(\mathcal {H},\eta))^\Gamma$ by $$(T_1 T_2)(x,y)=\sum\limits_{z\in X_d}T_1(x,z)T_2(z,y).$$
Let $$E'=\left\{\sum\limits_{x\in X_d}a_x[x]:a_x\in \mathcal {A}(H,\xi)\hat{\otimes}\mathcal {A}(\mathcal {H},\eta)\hat{\otimes} K,\sum\limits_{x\in X_d}a_x^*a_x\ \text{converge in norm}\right\}$$
Then $E'$ is a Hilbert module over $\mathcal {A}(H,\xi)\hat{\otimes}\mathcal {A}(\mathcal {H},\eta)\hat{\otimes}K$ with
$$\left\langle\sum\limits_{x\in X_d}a_x[x],\sum\limits_{x\in X_d}b_x[x]\right\rangle=\sum\limits_{x\in X_d}a_x^*b_x,$$
$$\left(\sum\limits_{x\in X_d}a_x[x]\right)a=\sum\limits_{x\in X_d}a_xa[x]$$
for all $a\in \mathcal {A}(H,\xi)\hat{\otimes}\mathcal {A}(\mathcal {H},\eta)\hat{\otimes} K$ and $\sum\limits_{x\in X_d}a_x[x]\in E'$.

$C^*_{alg}(P_d(X),\mathcal {A}(H,\xi)\hat{\otimes}\mathcal {A}(\mathcal {H},\eta))^{\Gamma}$ acts on $E$ by
$$T\left(\sum\limits_{x\in X_d}a_x[x]\right)=\sum\limits_{y\in X_d}\left(\sum\limits_{x\in X_d}T(y,x)a_x\right)[y],$$
where $T\in C^*_{alg}(P_d(X),\mathcal {A}(H,\xi)\hat{\otimes}\mathcal {A}(\mathcal {H},\eta))^{\Gamma}$ and $\sum\limits_{x\in X_d}a_x[x]\in E'$.
\begin{defn}
The {\it twisted equivariant Roe algebra} $C^*(P_d(X),\mathcal {A}(H,\xi)\hat{\otimes}\mathcal {A}(\mathcal {H},\eta))^{\Gamma}$ is defined to be the operator norm closure of $C^*_{alg}(P_d(X),\mathcal {A}(H,\xi)\hat{\otimes}\mathcal {A}(\mathcal {H},\eta))^{\Gamma}$ in $\mathcal {B}(E')$, the $C^*$-algebra of all module homomorphisms from $E'$ to $E'$ for which there is an adjoint module homomorphism.
\end{defn}

Let $C_{L,alg}^*(P_d(X),\mathcal {A}(H,\xi)\hat{\otimes}\mathcal {A}(\mathcal {H},\eta))^{\Gamma}$ be the set of all bounded and uniformly norm-continuous functions
$$g:\mathbb{R}_+\rightarrow C_{alg}^*(P_d(X),\mathcal {A}(H,\xi)\hat{\otimes}\mathcal {A}(\mathcal {H},\eta))^\Gamma$$
such that
\begin{enumerate}
  \item[(1)] $\sup\{d(x,y):g(t)(x,y)\neq 0\}\rightarrow 0 \ \text{as}\  t\rightarrow \infty$ with $x,y\in X_d$;

  \item[(2)] there exists $L>0$ such that for any $y\in P_d(X)$, $\sharp\{x:g(t)(x,y)\neq 0\}<L$, $\sharp\{x:g(t)(y,x)\neq 0\}<L$ for any $t\in \mathbb{R}_+$.
 \end{enumerate}

\begin{defn}
The {\it twisted $\Gamma$-equivariant localization algebra} $$C_L^*(P_d(X),\mathcal {A}(H,\xi)\hat{\otimes}\mathcal {A}(\mathcal {H},\eta))^{\Gamma}$$ is defined to be the norm completion of $C_{L,alg}^*(P_d(X),\mathcal {A}(H,\xi)\hat{\otimes}\mathcal {A}(\mathcal {H},\eta))^{\Gamma}$ with the norm defined by
$$\|g\|=\sup\limits_{t\in \mathbb{R}_+}\|g(t)\|_{C^*(P_d(X),\mathcal {A}(H,\xi)\hat{\otimes}\mathcal {A}(\mathcal {H},\eta))^\Gamma}.$$
\end{defn}

\begin{defn}
Let $C_{\pi,L,alg}^*(P_d(X),\mathcal {A}(H,\xi)\hat{\otimes}\mathcal {A}(\mathcal {H},\eta))^{\Gamma}$ be a $*$-algebra, which is defined by replacing condition (2) by $$\sup\{d(\pi(x),\pi(y)):g(t)(x,y)\neq 0\}\rightarrow 0 \ \text{as}\  t\rightarrow \infty$$ in the definition of $C_{L,alg}^*(P_d(X),\mathcal {A}(H,\xi)\hat{\otimes}\mathcal {A}(\mathcal {H},\eta))^{\Gamma}$.
The {\it twisted equivariant $\pi$-localization algebra} $C_{\pi,L}^*(P_d(X),\mathcal {A}(H,\xi)\hat{\otimes}\mathcal {A}(\mathcal {H},\eta))^{\Gamma}$ is the norm completion of $C_{\pi,L,alg}^*(P_d(X),\mathcal {A}(H,\xi)\hat{\otimes}\mathcal {A}(\mathcal {H},\eta))^{\Gamma}$ with the norm defined by
$$\|g\|=\sup\limits_{t\in \mathbb{R}_+}\|g(t)\|_{C^*(P_d(X),\mathcal {A}(H,\xi)\hat{\otimes}\mathcal {A}(\mathcal {H},\eta))^\Gamma}.$$
\end{defn}

It is easy to see that we have an inclusion map $$\iota: C_L^*(P_d(X), \mathcal {A}(H,\xi)\hat{\otimes}\mathcal {A}(\mathcal {H},\eta))^\Gamma\rightarrow C_{\pi,L}^*(P_d(X), \mathcal {A}(H,\xi)\hat{\otimes}\mathcal {A}(\mathcal {H},\eta))^{\Gamma},$$ which induces a homomorphism $\iota_*$ at the $K$-theory level.

We note that the $C^*$-algebras $C^*(P_d(X),\mathcal {A}(\mathcal {H},\eta))^{\Gamma}$, $C_L^*(P_d(X),\mathcal {A}(\mathcal {H},\eta))^{\Gamma}$ and $C_{\pi,L}^*(P_d(X),\mathcal {A}(\mathcal {H},\eta))^{\Gamma}$ can be defined in the similar way as above but with the matrix coefficient in $\mathcal {A}(\mathcal {H},\eta)\hat{\otimes}K$. We also have an inclusion map $$\tau: C_L^*(P_d(X), \mathcal {A}(\mathcal {H},\eta))^\Gamma\rightarrow C_{\pi,L}^*(P_d(X), \mathcal {A}(\mathcal {H},\eta))^{\Gamma},$$ which induces a homomorphism $\tau_*$ at the $K$-theory level.

\begin{thm} \label{sec:4:thm:4.8}
Let $X$ be a $\Gamma$-space with bounded distortion. If $X/\Gamma$ and $\Gamma$ are coarse embeddable into Hilbert spaces, then the following map
\begin{align*}
\iota_*:&\lim\limits_{d\rightarrow \infty}K_*(C_L^*(P_d(X), \mathcal {A}(H,\xi)\hat{\otimes}\mathcal {A}(\mathcal {H},\eta))^\Gamma)\\
&\hspace{55mm}\rightarrow \lim\limits_{d\rightarrow \infty}K_*(C_{\pi,L}^*(P_d(X), \mathcal {A}(H,\xi)\hat{\otimes}\mathcal {A}(\mathcal {H},\eta))^{\Gamma})
\end{align*}
is an isomorphism.
\end{thm}

To prove the above theorem, we need the following definitions. Recall that $\mathbb{R}_+\times H$ is endowed with the weakest topology for which the projection to $H$ is weakly continuous and the function $(t,w)\mapsto t^2+\|w\|^2$ is continuous. This topology makes $\mathbb{R}_+\times H$ into a locally compact Hausdorff space. For any open set $O\in \mathbb{R}_+\times H$, with the similar manner in section 3, we define $C_{alg}^*(P_d(X),\mathcal {A}(H,\xi)_O\hat{\otimes}\mathcal {A}(\mathcal {H},\eta))^{\Gamma}$ to be the subalgebra of $C_{alg}^*(P_d(X),\mathcal {A}(H,\xi)\hat{\otimes}\mathcal {A}(\mathcal {H},\eta))^{\Gamma}$ defined by
\begin{align*}
&C_{alg}^*(P_d(X),\mathcal {A}(H,\xi)_O\hat{\otimes}\mathcal {A}(\mathcal {H},\eta))^{\Gamma}:\\
&\hspace{10mm} =\{T\in C_{alg}^*(P_d(X),\mathcal {A}(H,\xi)\hat{\otimes}\mathcal {A}(\mathcal {H},\eta))^{\Gamma}: T(x,y)\text{\ is\ the\ finite\ sum \ of\ }\\&\hspace{10mm}T_1(x,y)\hat{\otimes} T_2(x,y)\hat{\otimes} k
\ \text{with}\ \supp(T_1(x,y))\subseteq O, \forall x,y\in X_d\}.
\end{align*}
The norm closure of $C_{alg}^*(P_d(X),\mathcal {A}(H,\xi)_O\hat{\otimes}\mathcal {A}(\mathcal {H},\eta))^{\Gamma}$ is defined to be the $C^*$-algebra  $C^*(P_d(X),\mathcal {A}(H,\xi)_O\hat{\otimes}\mathcal {A}(\mathcal {H},\eta))^{\Gamma}$. Moreover, we can define $$C_L^*(P_d(X),\mathcal {A}(H,\xi)_O\hat{\otimes}\mathcal {A}(\mathcal {H},\eta))^{\Gamma}\ \text{and}\ C_{\pi,L}^*(P_d(X),\mathcal {A}(H,\xi)_O\hat{\otimes}\mathcal {A}(\mathcal {H},\eta))^{\Gamma}$$ by similar method as in section 3.

\begin{proof}[{\bf Proof of Theorem \ref{sec:4:thm:4.8}.}] For any $r>0$, let $O_r=\bigcup\limits_{\pi(x)\in X/\Gamma}B_{\mathbb{R}_+\times H }(\xi(\pi(x)),r)$, where $B_{\mathbb{R}_+\times H }(\xi(\pi(x)),r)
:=\left\{(s,h)\in\mathbb{R}_+\times H:s^2+\|h-\xi(\pi(x))\|^2<r^2\right\}$. It is sufficient to prove that  \begin{align*}
\iota_*:&\lim\limits_{d\rightarrow \infty}K_*(C_L^*(P_d(X), \mathcal {A}(H,\xi)_{O_r}\hat{\otimes}\mathcal {A}(\mathcal {H},\eta))^\Gamma)\\
&\hspace{55mm}\rightarrow \lim\limits_{d\rightarrow \infty}K_*(C_{\pi,L}^*(P_d(X), \mathcal {A}(H,\xi)_{O_r}\hat{\otimes}\mathcal {A}(\mathcal {H},\eta))^{\Gamma})
\end{align*}
is an isomorphism.

Let $Y$ be a fundamental domain such that the $\Gamma$-action on $X$ has bounded distortion. Since $X$ has bounded geometry and $\Gamma$ acts on $X$ properly and isometrically, then $\widetilde{X}=X/\Gamma$ is a bounded geometrical metric space. So there exists finitely many mutually disjoint subsets of $\widetilde{X}$, say $\widetilde{X}_k:=\{\pi(x_j):j\in J_k, x_j\in Y\}$ with some index set $J_k$ for $k=1,2,\cdots, k_0$, such that $\widetilde{X}=\bigsqcup\limits_{k=1}^{k_0}\widetilde{X_k}$ and for $\pi(x_i),\pi(x_j)\in \widetilde{X}_k$, $d(\pi(x_i),\pi(x_j))$ is large enough such that $d(\xi(\pi(x_i)), \xi(\pi(x_j)))> 2r$, where $\pi:X\rightarrow X/\Gamma$ is the quotient map.

For each $k=1,2,\cdots, k_0$, we denote
$$O_{r,k}=\bigcup\limits_{j\in J_k}\big(B(\xi(\pi(x_j)), r)\big).$$
We will complete the proof of this theorem once we proved that
\begin{align*}
&\lim\limits_{d\rightarrow \infty}K_*(C_{L}^*(P_d(X),\mathcal {A}(H,\xi)_{O_{r,k}}\hat{\otimes}\mathcal {A}(\mathcal {H},\eta))^\Gamma)\\
&\hspace{48mm}\cong \lim\limits_{d\rightarrow \infty}K_*(C_{\pi,L}^*(P_d(X),\mathcal {A}(H,\xi)_{O_{r,k}}\hat{\otimes}\mathcal {A}(\mathcal {H},\eta))^\Gamma) \ \ \ \ (\ast)
\end{align*}
by using the Mayer-Vietoris sequence argument as Theorem 6.8 in~\cite{Yu}.

Since $\mathcal {A}(\mathcal {H},\eta)$ is a $\Gamma$-proper $C^*$-algebra, so it can be a direct limit of ideals $\mathcal {A}_n$, each of which is $\Gamma$-proper over some cocompact $\Gamma$-space. For each $n$, suppose $\mathcal{A}_n$ is $\Gamma$-proper over the cocompact $\Gamma$-space $W$. Since $\Gamma$ acts on $W$ properly and cocompactly, so $W$ can be covered by finite $\Gamma$-spaces of the form $\Gamma\times_{F_i} S_i$, where $S_i$ is a bounded subset of $W$, $F_i$ is a finite subgroup of $\Gamma$ and $1\leq i\leq n_0$. So to prove ($\ast$), it is sufficient to prove that for each $n$,
$$\lim\limits_{d\rightarrow \infty}K_*(C_{L}^*(P_d(X),\mathcal {A}(H,\xi)_{O_{r,k}}\hat{\otimes}\mathcal {A}_n)^\Gamma)
\cong \lim\limits_{d\rightarrow \infty}K_*(C_{\pi,L}^*(P_d(X),\mathcal {A}(H,\xi)_{O_{r,k}}\hat{\otimes}\mathcal {A}_n)^\Gamma).$$
By using the Mayer-Vietoris sequence argument, we only need to prove that for each $1\leq i\leq n_0$,
\begin{align*}
&\lim\limits_{d\rightarrow \infty}K_*(C_{L}^*(P_d(X),\mathcal {A}(H,\xi)_{O_{r,k}}\hat{\otimes}(\mathcal {A}_n)_{\Gamma\times_{F_i} S_i})^\Gamma)\\
&\hspace{50mm}\rightarrow \lim\limits_{d\rightarrow \infty}K_*(C_{\pi,L}^*(P_d(X),\mathcal {A}(H,\xi)_{O_{r,k}}\hat{\otimes}(\mathcal {A}_n)_{\Gamma\times_{F_i} S_i})^\Gamma)
\end{align*} is isomorphic.

Let $Y_k=\{x_j\in Y: j\in J_k,\pi(x_j)\in \widetilde{X}_k\}$. For any $R>0$, let 
$$\triangle'_j(R):=\{x\in P_d(X):\exists x_j\in Y_k,d(x,x_j)\leq R\}.$$ 
Let $$\prod^u\limits_{j\in J_k}C_{alg}^*(\Gamma\cdot\triangle'_j(R), \mathcal {A}(H,\xi)_{O_{r,k}}\hat{\otimes}(\mathcal {A}_n)_{\Gamma\times_{F_i} S_i})^\Gamma)$$ be the $*$-algebra such that for any $$\bigoplus\limits_{j\in J_k} T_j\in \prod^u\limits_{j\in J_k}C_{alg}^*(\Gamma\cdot\triangle'_j(R), \mathcal {A}(H,\xi)_{O_{r,k}}\hat{\otimes}(\mathcal {A}_n)_{\Gamma\times_{F_i} S_i})^\Gamma),$$ the propagations of $T_j$ are uniformly bounded. Define $$\prod^u\limits_{j\in J_k} C_L^*(\Gamma\cdot\triangle'_j(R), \mathcal {A}(H,\xi)_{O_{r,k}}\hat{\otimes}(\mathcal {A}_n)_{\Gamma\times_{F_i} S_i})^\Gamma)$$ to be the supremum norm closure of the algebra $$\Big\{\bigoplus\limits_{j\in J_k}b_j:[0,+\infty)\rightarrow \prod^u\limits_{j\in J_k}C_{alg}^*(\Gamma\cdot\triangle'_j(R), \mathcal {A}(H,\xi)_{O_{r,k}}\hat{\otimes}(\mathcal {A}_n)_{\Gamma\times_{F_i} S_i})^\Gamma)\Big\}$$
such that $\bigoplus\limits_{j\in J_k}b_j$ is bounded, uniformly continuous  and $\sup\limits_{j\in J_k}\{d(x,y):b_j(t)(x,y)\neq 0\}\rightarrow 0$ as $t\rightarrow \infty$.
Using the similar argument as Lemma 6.4 in \cite{Yu}, we can prove
\begin{align*}
&C_{L}^*(P_d(X),\mathcal {A}(H,\xi)_{O_{r,k}}\hat{\otimes}(\mathcal {A}_n)_{\Gamma\times_{F_i} S_i})^\Gamma\\
&\hspace{50mm}\cong\lim\limits_{R\rightarrow \infty}\prod^u\limits_{j\in J_k} C_L^*(\Gamma\cdot\triangle'_j(R), \mathcal {A}(H,\xi)_{O_{r,k}}\hat{\otimes}(\mathcal {A}_n)_{\Gamma\times_{F_i} S_i})^\Gamma.
\end{align*}
Since $\mathcal {A}_n$ is proper over $\Gamma\times_{F_i} S_i$, i.e. $C_0(\Gamma\times_{F_i} S_i)\mathcal {A}_n$ is dense in $\mathcal {A}_n$, 
it is not very difficult to prove that
\begin{align*}
&\lim\limits_{R\rightarrow \infty}\prod^u\limits_{j\in J_k} C_L^*(\Gamma\cdot\triangle'_j(R), \mathcal {A}(H,\xi)_{O_{r,k}}\hat{\otimes}(\mathcal {A}_n)_{\Gamma\times_{F_i} S_i})^\Gamma\\
&\hspace{50mm}\cong\lim\limits_{R\rightarrow \infty}\prod^u\limits_{j\in J_k} C_L^*(\widetilde{F_i}\cdot\triangle'_j(R), \mathcal {A}(H,\xi)_{O_{r,k}}\hat{\otimes}(\mathcal {A}_n)_{S_i})^{F_i}
\end{align*} 
in a way similar to the proof of Lemma 5.16 in \cite{Ruffus2}, where $\widetilde{F_i}$ is a finite set of $\Gamma$.

Since the $\Gamma$-action on $X$ has bounded distorsion, it is easy to prove that for any $R>0$, $\widetilde{F_i}\cdot\triangle'_j(R)$ is uniformly bounded for $j\in J_k$. If we let $\triangle_j$ be the simplex spanned by $F_i\cdot x_j$ with $x_j\in Y$, $x_j^0$ be the barycenter of $\triangle_j$. Let $\triangle_j(S)$ be the simplex with vertices $\{x\in P_d(X):d(x,x^0_j)\leq S\}$ for $d>S$, then
\begin{align*}
&\lim\limits_{d\rightarrow \infty}C_{L}^*(P_d(X),\mathcal {A}(H,\xi)_{O_{r,k}}\hat{\otimes}(\mathcal {A}_n)_{\Gamma\times_{F_i} S_i})^\Gamma\\
&\hspace{50mm}\cong\lim\limits_{S\rightarrow \infty}\prod^u\limits_{j\in J_k} C_L^*(\triangle_j(S), \mathcal {A}(H,\xi)_{O_{r,k}}\hat{\otimes}(\mathcal {A}_n)_{S_i})^{F_i}.
\end{align*}
Similarly we can prove
\begin{align*}
&\lim\limits_{d\rightarrow \infty}C_{\pi,L}^*(P_d(X),\mathcal {A}(H,\xi)_{O_{r,k}}\hat{\otimes}(\mathcal {A}_n)_{\Gamma\times_{F_i} S_i})^\Gamma\\
&\hspace{50mm}\cong\lim\limits_{S\rightarrow \infty}\prod^u\limits_{j\in J_k} C_{\pi,L}^*(\triangle_j(S), \mathcal {A}(H,\xi)_{O_{r,k}}\hat{\otimes}(\mathcal {A}_n)_{S_i})^{F_i}.
\end{align*}

When $S$ is big enough, $\triangle_j\subseteq \triangle_j(S)$ and $\triangle_j(S)$ is $F_i$-invariant for any $j\in J_k$. So $\{\triangle_j(S)\}_{j\in J_k}$ is $F_i$-equivariant strong Lipschitz homotopy equivalent to $\{x_j^0\}_{j\in J_k}$(\cite{Yu}, P. 218-219). With the essentially same proof of the non-equivariant case(\cite{Yu}, Lemma 6.5), we have
\begin{align*}
&K_*(\prod^u\limits_{j\in J_k} C_{L,0}^*(\triangle_j(S), \mathcal {A}(H,\xi)_{O_{r,k}}\hat{\otimes}(\mathcal {A}_n)_{S_i})^{F_i})\\
&\hspace{60mm}\cong K_*(\prod^u\limits_{j\in J_k}C_{L,0}^*(x^0_j, \mathcal {A}(H,\xi)_{O_{r,k}}\hat{\otimes}(\mathcal {A}_n)_{S_i})^{F_i}).
\end{align*}
However, by the Eilenberg swindle argument, we know that $$K_*(\prod^u\limits_{j\in J_k}C_{L,0}^*(x^0_j, \mathcal {A}(H,\xi)_{O_{r,k}}\hat{\otimes}(\mathcal {A}_n)_{S_i})^{F_i})=0.$$ Thus $$K_*(\prod^u\limits_{j\in J_k} C_{L,0}^*(\triangle_j(S), \mathcal {A}(H,\xi)_{O_{r,k}}\hat{\otimes}(\mathcal {A}_n)_{S_i})^{F_i})=0.$$
Similarly, we can prove that for any $S>0$ which is sufficient big, $$K_*(\prod^u\limits_{j\in J_k} C_{\pi,L,0}^*(\triangle_j(S), \mathcal {A}(H,\xi)_{O_{r,k}}\hat{\otimes}(\mathcal {A}_n)_{S_i})^{F_i})=0.$$
By the six term exact sequence of $C^*$-algebra K-theory, we know that
\begin{align*}
&\lim\limits_{S\rightarrow \infty}K_*( \prod^u\limits_{j\in J_k} C_{L}^*(\triangle_j(S), \mathcal {A}(H,\xi)_{O_{r,k}}\hat{\otimes}(\mathcal {A}_n)_{S_i})^{F_i})\\
&\hspace{50mm}\cong \lim\limits_{S\rightarrow \infty}K_*(\prod^u\limits_{j\in J_k} C^*(\triangle_j(S), \mathcal {A}(H,\xi)_{O_{r,k}}\hat{\otimes}(\mathcal {A}_n)_{S_i})^{F_i}),
\end{align*}
\begin{align*}
&\lim\limits_{S\rightarrow \infty}K_*( \prod^u\limits_{j\in J_k} C_{\pi,L}^*(\triangle_j(S), \mathcal {A}(H,\xi)_{O_{r,k}}\hat{\otimes}(\mathcal {A}_n)_{S_i})^{F_i})\\
&\hspace{50mm}\cong \lim\limits_{S\rightarrow \infty}K_*(\prod^u\limits_{j\in J_k} C^*(\triangle_j(S), \mathcal {A}(H,\xi)_{O_{r,k}}\hat{\otimes}(\mathcal {A}_n)_{S_i})^{F_i}).
\end{align*}
So we have proved that
\begin{align*}
&\lim\limits_{S\rightarrow \infty}K_*( \prod^u\limits_{j\in J_k} C_{L}^*(\triangle_j(S), \mathcal {A}(H,\xi)_{O_{r,k}}\hat{\otimes}(\mathcal {A}_n)_{S_i})^{F_i})\\
&\hspace{50mm}\cong \lim\limits_{S\rightarrow \infty}K_*( \prod^u\limits_{j\in J_k} C_{\pi,L}^*(\triangle_j(S), \mathcal {A}(H,\xi)_{O_{r,k}}\hat{\otimes}(\mathcal {A}_n)_{S_i})^{F_i}).
\end{align*}
Thus for each $1\leq i\leq n_0$, we have
\begin{align*}
&\lim\limits_{d\rightarrow \infty}K_*(C_{L}^*(P_d(X),\mathcal {A}(H,\xi)_{O_{r,k}}\hat{\otimes}(\mathcal {A}_n)_{\Gamma\times_{F_i}S_i})^\Gamma)\\
&\hspace{50mm}\cong \lim\limits_{d\rightarrow \infty}K_*(C_{\pi,L}^*(P_d(X),\mathcal {A}(H,\xi)_{O_{r,k}}\hat{\otimes}(\mathcal {A}_n)_{\Gamma\times_{F_i}S_i})^\Gamma).
\end{align*}
So we have finished the proof of the theorem.
\end{proof}

Now we are ready to prove the following theorem:
\begin{thm} \label{sec:4:thm:4.10}
Let $X$ be a $\Gamma$-space with bounded distortion. If $X/\Gamma$ and $\Gamma$ are coarse embeddable into Hilbert spaces, then the following map
$$
\tau_*:\lim\limits_{d\rightarrow \infty}K_*(C_L^*(P_d(X), \mathcal {A}(\mathcal {H},\eta))^\Gamma)
\rightarrow \lim\limits_{d\rightarrow \infty}K_*(C_{\pi,L}^*(P_d(X), \mathcal {A}(\mathcal {H},\eta))^{\Gamma})
$$
is injective.
\end{thm}
\begin{proof}
For any $d>0$, choose a countable $\Gamma$-invariant dense subset $X_d$ of $P_d(X)$ for each $d>0$ such that $X_{d_1}\subseteq X_{d_2}$ if $d_1\leq d_2$. For each $t\geq 1$, define a map
$$\beta_t:\mathcal {S}\hat{\otimes}C_{alg}^*(P_d(X), \mathcal {A}(\mathcal {H},\eta))^\Gamma\rightarrow C^*(P_d(X),\mathcal {A}(H,\xi)\hat{\otimes}\mathcal {A}(\mathcal {H},\eta))^\Gamma$$ by$$(\beta_t(g\hat{\otimes }T))(x,y)=(\beta(\pi(x)))(g_t)\hat{\otimes}T(x,y)$$ for all $g\in \mathcal {S}$ and $T\in C_{alg}^*(P_d(X), \mathcal {A}(\mathcal {H},\eta))^\Gamma$, where $g_t(s)=g(t^{-1}s)$ for all $s\in \mathbb{R}$, and $\beta(\pi(x)):\mathcal {S}=\mathcal {A}(\xi(\pi(x)))\rightarrow \mathcal {A}(H,\xi)$ is the $*$-homomorphism from the Definition \ref{clif-alg}.

Similarly, we can define
$$(\beta_{L})_t:\mathcal {S}\hat{\otimes}C_{L,alg}^*(P_d(X), \mathcal {A}(\mathcal {H},\eta))^\Gamma\rightarrow C_{L}^*(P_d(X),\mathcal {A}(H,\xi)\hat{\otimes}\mathcal {A}(\mathcal {H},\eta))^\Gamma,$$
$$(\beta_{\pi,L})_t:\mathcal {S}\hat{\otimes}C_{\pi,L,alg}^*(P_d(X), \mathcal {A}(\mathcal {H},\eta))^\Gamma\rightarrow C_{\pi,L}^*(P_d(X),\mathcal {A}(H,\xi)\hat{\otimes}\mathcal {A}(\mathcal {H},\eta))^\Gamma.$$
The maps $\beta_t$, $(\beta_{L})_t$ and $(\beta_{\pi,L})_t$ can be extended to asymptotic morphisms (cf. Lemma 7.6 in~\cite{Yu} or Lemma 5.8 in~\cite{Fu})
$$\beta: \mathcal {S}\hat{\otimes}C^*(P_d(X), \mathcal {A}(\mathcal {H},\eta))^\Gamma \rightarrow C^*(P_d(X), \mathcal {A}(H,\xi)\hat{\otimes}\mathcal {A}(\mathcal {H},\eta))^\Gamma;$$
$$\beta_L: \mathcal {S}\hat{\otimes}C_L^*(P_d(X), \mathcal {A}(\mathcal {H},\eta))^\Gamma \rightarrow C_L^*(P_d(X), \mathcal {A}(H,\xi)\hat{\otimes}\mathcal {A}(\mathcal {H},\eta))^\Gamma;$$
$$\beta_{\pi,L}: \mathcal {S}\hat{\otimes}C_{\pi,L}^*(P_d(X), \mathcal {A}(\mathcal {H},\eta))^\Gamma \rightarrow C_{\pi,L}^*(P_d(X), \mathcal {A}(H,\xi)\hat{\otimes}\mathcal {A}(\mathcal {H},\eta))^\Gamma.$$
So they induce the homomorphisms $$(\beta)_*: K_*(C^*(P_d(X), \mathcal {A}(\mathcal {H},\eta))^\Gamma) \rightarrow K_*(C^*(P_d(X), \mathcal {A}(H,\xi)\hat{\otimes}\mathcal {A}(\mathcal {H},\eta))^\Gamma);$$
 $$(\beta_L)_*: K_*(C_L^*(P_d(X), \mathcal {A}(\mathcal {H},\eta))^\Gamma) \rightarrow K_*(C_L^*(P_d(X), \mathcal {A}(H,\xi)\hat{\otimes}\mathcal {A}(\mathcal {H},\eta))^\Gamma);$$
$$(\beta_{\pi,L})_*: K_*(C_{\pi,L}^*(P_d(X), \mathcal {A}(\mathcal {H},\eta))^\Gamma) \rightarrow K_*(C_{\pi,L}^*(P_d(X), \mathcal {A}(H,\xi)\hat{\otimes}\mathcal {A}(\mathcal {H},\eta))^\Gamma).$$
And we have the following commutative diagram:$$\xymatrix{
 \lim\limits_{d\rightarrow \infty}K_*(C_L^*(P_d(X), \mathcal {A}(\mathcal {H},\eta))^\Gamma)\ar[d]_{\tau_*} \ar[r]^{(\beta_L)_*\ \ \ \ \ \ }  & \lim\limits_{d\rightarrow \infty}K_*(C_L^*(P_d(X), \mathcal {A}(H,\xi)\hat{\otimes}\mathcal {A}(\mathcal {H},\eta))^\Gamma) \ar[d]_{\iota_*}  \\
  \lim\limits_{d\rightarrow \infty}K_*(C_{\pi,L}^*(P_d(X), \mathcal {A}(\mathcal {H},\eta))^\Gamma)  \ar[r]^{(\beta_{\pi,L})_*\ \ \ \ \ \ } &    \lim\limits_{d\rightarrow \infty}K_*(C_{\pi,L}^*(P_d(X), \mathcal {A}(H,\xi)\hat{\otimes}\mathcal {A}(\mathcal {H},\eta))^\Gamma).}$$
It is not difficult to prove that $(\beta_L)_*$ is an isomorphism by the Mayer-Vietoris sequence and five lemma argument and induction on the dimension of the skeletons of $P_d(X)$, together with the Bott periodicity. Since $\iota_*$ is also an isomorphism by Theorem \ref{sec:4:thm:4.8}, so $\tau_*$ is injective.
\end{proof}

\section{The Main Theorem}

In this section, we shall introduce a Bott map from the $K$-theory of the equivariant localization algebra $C_L^*(P_d(X), C(Z))^\Gamma$  to the $K$-theory of the equivariant twisted localization algebra $C_L^*(P_d(X), \mathcal {A}(\mathcal {H}))^\Gamma$, where $Z$ and $\mathcal {A}(\mathcal {H})$ are defined in Section 4. We prove that this map is an isomorphism. This, together with the main result in section 3, implies the main theorem.

For any $d\geq 0$, choose a countable $\Gamma$-invariant dense subset $X_d$ of $P_d(X)$ for each $d>0$ such that $X_{d_1}\subseteq X_{d_2}$ if $d_1\leq d_2$. Let $Y_d$ be a fundamental domain of $X_d$. Let $C^*(P_d(X),C(Z))^\Gamma$ be a $C^*$-algebra defined similarly to Roe algebra $C^*(P_d(X))^\Gamma$ but with $C(Z)\hat{\otimes }K(H)$ as the matrix coefficients.
For each $t\geq 1$, define a map
$$\beta_t:\mathcal {S}\hat{\otimes}C_{alg}^*(P_d(X),C(Z))^\Gamma\rightarrow C^*(P_d(X),\mathcal {A}(\mathcal {H}))^\Gamma$$ by
$$\left[(\beta_t(g\hat{\otimes }T))(x,y)\right](z)=\frac{1}{|\Gamma_x|}\sum\limits_{\gamma\in \Gamma_x}\gamma\cdot(\beta^z (g_t)\hat{\otimes}T(x,y)(z))$$
with $x\in Y_d, y\in X_d$ such that
$$\left[(\beta_t(g\hat{\otimes }T))(\gamma x,\gamma y)\right](\gamma z)=\gamma\cdot(\beta^z (g_t)\hat{\otimes}T(x,y)(z))$$
for all $g\in \mathcal {S}$, $T\in C_{alg}^*(P_d(X), C(Z))^\Gamma$, where $|\Gamma_x|$ is the element number of the stabiliser $\Gamma_x$,  $g_t(s)=g(t^{-1}s)$ for all $s\in \mathbb{R}$, and $\beta^z:\mathcal {S}\rightarrow \mathcal {A}(H_z)$ is the $*$-homomorphism defined by $f(x)\mapsto f(X\hat{\otimes}1+1\hat{\otimes}C)$ for $f(x)\in \mathcal {S}=C_0(\mathbb{R})$, where $X$ and $C$ should be viewed as unbounded degree one multiplier on
$\mathcal {S}$ and $\mathcal {A}(H_z)$, $f(X\hat{\otimes}1+1\hat{\otimes}C)$ is defined by functional calculus. With the continuous field structure of
$(H_z)_{z\in Z}$ and $\mathcal {A}(\mathcal {H})$,
there is a homomorphism $\beta: C(Z)\hat{\otimes}\mathcal {S}\longrightarrow \mathcal {A}(\mathcal {H})$ which is defined by using $\beta^z$ pointwisely.
It induces an isomorphism $\beta_*:K_*(C(Z)\hat{\otimes}\mathcal {S})\rightarrow K_*(\mathcal {A}(\mathcal {H}))$, which is a field version of the Higson-Kasparov-Trout's Bott Periodicity Theorem (see~\cite{HKT} or~\cite{Tu}).
%More introduction about this isomorphism can be founded in \cite{}.

Applying the above maps pointwise, we can define $$(\beta_L)_t:\mathcal {S}\hat{\otimes}C_{L,alg}^*(P_d(X),C(Z))^\Gamma\rightarrow C_L^*(P_d(X),\mathcal {A}(\mathcal {H}))^\Gamma.$$ We know that $\beta_t$, $(\beta_L)_t$ extend to asymptotic morphisms
$$\beta: \mathcal {S}\hat{\otimes}C^*(P_d(X),C(Z))^\Gamma \rightarrow C^*(P_d(X), \mathcal {A}(\mathcal {H}))^\Gamma;$$
$$\beta_L: \mathcal {S}\hat{\otimes}C_L^*(P_d(X),C(Z))^\Gamma \rightarrow C_L^*(P_d(X), \mathcal {A}(\mathcal {H}))^\Gamma.$$ Then they induce homomorphisms at $K$-theory level
$$\beta_*:K_*(C^*(P_d(X),C(Z))^\Gamma)\rightarrow K_*(C^*(P_d(X), \mathcal {A}(\mathcal {H}))^\Gamma)$$ and $$(\beta_L)_*:K_*(C_L^*(P_d(X),C(Z))^\Gamma)\rightarrow K_*(C_L^*(P_d(X), \mathcal {A}(\mathcal {H}))^\Gamma).$$ Similarly, we can define a homomorphism $$(\beta_{\pi,L})_*:K_*(C_{\pi,L}^*(P_d(X),C(Z))^\Gamma)\rightarrow K_*(C_{\pi,L}^*(P_d(X), \mathcal {A}(\mathcal {H}))^\Gamma).$$

\begin{thm}\label{sec:5:thm:5.1}
For any $d>0$, $$(\beta_L)_*:K_*(C_L^*(P_d(X),C(Z))^\Gamma)\rightarrow K_*(C_L^*(P_d(X), \mathcal {A}(\mathcal {H},\eta))^\Gamma)$$ is an isomorphism.
\end{thm}

\begin{proof}
By the Mayer-Vietoris sequence and five lemma argument and induction on the dimension of the skeletons of $P_d((X)$, it is sufficient to prove the $0$-dimensional case. In the $0$-dimensional case, the isomorphism follows from the the Bott periodicity.
\end{proof}

\begin{proof}[{\bf Proof of Theorem \ref{sec:1:thm:1.2}.}]

For any $d>0$, let $$i_*:K_*(C_L^*(P_d(X))^{\Gamma})\rightarrow K_*(C_L^*(P_d(X), C(Z))^{\Gamma})$$ be the homomorphism induced by the inclusion map $i:\mathbb{C}\rightarrow C(Z)$. Then $i_*$ is an isomorphism by using the Mayer-Vietoris sequence and five lemma argument by the induction on the dimension of the skeletons of $P_d(X)$.

Let $\widetilde{\tau}: C_L^*(P_d(X))^\Gamma\rightarrow C_{\pi,L}^*(P_d(X))^{\Gamma}$ be the inclusion map. It induces a homomorphism $\widetilde{\tau}_*: K_*(C_L^*(P_d(X))^\Gamma)\rightarrow K_*(C_{\pi,L}^*(P_d(X))^{\Gamma})$. We get the following commutative diagram:

$$
\xymatrix{
  \lim\limits_{d\rightarrow \infty}K_*(C_L^*(P_d(X))^\Gamma) \ar[d]_{i_*} \ar[r]^{\widetilde{\tau}_*} & \lim\limits_{d\rightarrow \infty}K_*(C_{\pi,L}^*(P_d(X))^{\Gamma}) \ar[d]^{(i')_*} \\
  \lim\limits_{d\rightarrow \infty}K_*(C_L^*(P_d(X), C(Z))^\Gamma) \ar[d]_{(\beta_L)_*} \ar[r] & \lim\limits_{d\rightarrow \infty}K_*(C_{\pi,L}^*(P_d(X), C(Z))^{\Gamma}) \ar[d]^{(\beta_{\pi,L})_*} \\
  \lim\limits_{d\rightarrow \infty}K_*(C_L^*(P_d(X), \mathcal {A}(\mathcal {H},\eta))^\Gamma) \ar[r]^{\tau_*} & \lim\limits_{d\rightarrow \infty}K_*(C_{\pi,L}^*(P_d(X), \mathcal {A}(\mathcal {H},\eta))^\Gamma)   }
$$
where $i'_*:K_*(C_{\pi,L}^*(P_d(X))^{\Gamma})\rightarrow K_*(C_{\pi,L}^*(P_d(X), C(Z))^{\Gamma})$ induced by the inclusion map $i:\mathbb{C}\rightarrow C(Z)$. From Theorem \ref{sec:5:thm:5.1} and Theorem \ref{sec:4:thm:4.10}, we know that $\tau_*\circ (\beta_L)_*\circ i_*$ is injective. So $\widetilde{\tau}_*$ is an injectivity.
By another commutative diagram
$$
\xymatrix{
                &         \lim\limits_{d\rightarrow \infty}K_*(C_{\pi,L}^*(P_d(X))^{\Gamma}) \ar[d]^{(e_\pi)_*}     \\
  \lim\limits_{d\rightarrow \infty}K_*(C_L^*(P_d(X))^\Gamma) \ar[ur]^{\widetilde{\tau}_*} \ar[r]_{e_*} & \lim\limits_{d\rightarrow \infty}K_*(C^*(P_d(X))^{\Gamma}),            }
$$
we know that $e_*=(e_\pi)_*\circ \widetilde{\tau}_*$ is an injectivity since $(e_\pi)_*$ is injective by Theorem \ref{sec:3:thm:3.12}.

\end{proof}
\par\noindent{\bf Acknowledgement.} The authors wish to thank the referees for many valuable and constructive suggestions. They also  
would like to thank Jintao Deng for many useful discussions, and for carefully reading the articles and making many helpful comments.

\bigskip

\footnotesize

\noindent  Benyin Fu \\
College of Statistics and Mathematics,\\
Shanghai Lixin University of Accounting and Finance\\
Shanghai 201209, P. R. China.\\
E-mail: \url{fuby@lixin.edu.cn}\\

\noindent  Xianjin Wang \\
College of Mathematics and Statistics,\\
Chongqing University, (at Huxi Campus),\\
Chongqing 401331, P. R. China.\\
E-mail: \url{xianjinwang@cqu.edu.cn}\\

\noindent Guoliang Yu \\
Department of Mathematics,\\
Texas A\&M University, TX, 77843-3368, USA\\
E-mail: \url{guoliangyu@math.tamu.edu}\\

\end{document}